\documentclass[preprint,onefignum,onetabnum,pdftex]{siamart171218}
\usepackage{cool} 
\usepackage{placeins}
\usepackage{color}
\usepackage[countmax]{subfloat}
\usepackage{float}
\usepackage{caption}
\usepackage{placeins} 
\usepackage{subfig}
\usepackage{standalone}
\usepackage{inputenc}
\usepackage{microtype}
\usepackage{graphicx}
\usepackage{bm}

\graphicspath{{./figure/}{./riccardo/Figures/}}
\usepackage{tikz}
\usepackage{pgfplots}
 \pgfplotsset{compat=1.13}
\usepackage{pgfplotstable}
\usetikzlibrary{calc}
\definecolor{color2}{rgb}{0.872549019607843,0.019607843137255,0.0819607843137255}
\definecolor{dred}{rgb}{0.92,0,0}

\newcommand\restr[2]{{
  \left.\kern-\nulldelimiterspace 
  #1 
  \vphantom{\big|} 
  \right|_{#2} 
  }} 
  \newcommand{\abs}[1]{\left\lvert#1\right\rvert}
\newcommand{\norm}[1]{\left\lVert#1\right\rVert}
\let\hat\widehat
\let\tilde\widetilde
\usepackage{lipsum}
\usepackage{amsfonts}
\usepackage{graphicx}
\usepackage{epstopdf}
\usepackage{algorithmic}
\ifpdf
  \DeclareGraphicsExtensions{.eps,.pdf,.png,.jpg}
\else
  \DeclareGraphicsExtensions{.eps}
\fi

\newcommand{\RN}[1]{%
  \textup{\uppercase\expandafter{\romannumeral#1}}%
}

\newtheorem{remark}{Remark}

\headers{Overlapping Multi-Patch Isogeometric Method}{P. Antolin, A. Buffa, R. Puppi, and X. Wei}

\newcommand{\pd}[2]{\frac{\partial #1}{\partial #2}}

\begin{document}

\title{Overlapping Multi-Patch Isogeometric Method with Minimal Stabilization\thanks{Submitted.
\funding{The authors were partially supported by ERC AdG project CHANGE n. 694515.}}}
\date{November 2019}

\author{ Pablo Antolin\thanks{Chair of Modelling and Numerical Simulation, Institute of Mathematics, \'Ecole Polytechnique F\'ed\'erale de Lausanne, Lausanne, Switzerland (\email{pablo.antolin@epfl.ch}, \email{annalisa.buffa@epfl.ch}, \email{riccardo.puppi@epfl.ch}, \email{xiaodong.wei@epfl.ch}).}
\and Annalisa Buffa\footnotemark[2]\ \thanks{ Istituto di Matematica Applicata e Tecnologie Informatiche "Enrico Magenes" del CNR, Pavia, Italy.}
\and Riccardo Puppi\footnotemark[2]
 \and Xiaodong Wei\footnotemark[2]}.
\maketitle

\begin{abstract}
We present a novel method for isogeometric analysis (IGA) to directly work on geometries constructed by Boolean operations including difference (i.e., trimming), union and intersection. Particularly, this work focuses on the union operation, which involves multiple independent, generally non-conforming and trimmed spline patches. Given a series of patches, we overlay one on top of another in a certain order. While the invisible part of each patch is trimmed away, the visible parts of all the patches constitute the entire computational domain. We employ the Nitsche's method to weakly couple independent patches through visible interfaces. Moreover, we propose a minimal stabilization method to address the instability issue that arises on the interfaces shared by small trimmed elements. We show in theory that our proposed method recovers stability and guarantees well-posedness of the problem as well as optimal error estimates. In the end, we numerically verify the theory by solving the Poisson's equation on various geometries that are obtained by the union operation.
 \end{abstract}



\sloppy

\section{Introduction}

Isogeometric analysis (IGA) was proposed to tightly integrate computer-aided design (CAD) and engineering simulation by utilizing basis functions in CAD to approximate the solution in simulation \cite{ref:hughes05, ref:cottrell09}. NURBS (non-uniform rational B-spline), the de facto industry standard of CAD, is the first choice of bases explored in IGA and it has shown many numerical advantages in terms of accuracy \cite{ref:evans09}, stability and robustness \cite{ref:lipton10}.  

While IGA has advanced in various aspects, constructing analysis-suitable geometric representations remains one of the most challenging problems because geometric representations in current CAD systems generally are not ready for analysis. On one hand, a volume description is generally needed to serve as the computational domain for analysis but standard CAD systems only employ a boundary representation (B-rep) for solid modeling without containing any information about the interior domain \cite{ref:cohen01}. On the other hand, Boolean operations, including trimming (i.e., difference), union and intersection, are universal in CAD systems to create complex free-form surfaces from simple primitives. As a consequence, a B-rep is a collection of trimmed spline surfaces with possible gaps and overlaps, which again pose a barrier to analysis because analysis requires a watertight geometric representation. 

Numerous efforts have been devoted to constructing volume descriptions and supporting Boolean operations. Constructing a volume description from a given B-rep, namely volumetric parameterization, seeks to create a trivariate geometric mapping. Representative methods include those based on B-splines \cite{ref:martin09, ref:xu13, ref:chen19} and T-splines \cite{ref:zhang12, ref:zhang13, ref:wang13, ref:liu14}. Patch gluing is another alternative to construct complex geometries and can be obtained in a variety of ways, such as the IETI (IsogEometric Tearing and Interconnecting) method \cite{ref:kleiss12} and the mortar method \cite{ref:brivadis15, ref:seitz16}. Moreover, it often relies on mesh intersection tools \cite{ref:gander09} to handle non-conforming patches.


Along the direction of supporting Boolean operations in IGA, most researches focus on shell techniques \cite{ref:schmidt12, ref:breitenberger15, ref:guo18, ref:teschemacher18, ref:leidinger19a} as B-reps naturally have a shell structure. The related major challenges include performing numerical integration for trimmed elements \cite{ref:kim09, ref:nagy15} and handling the stability issue caused by small trimmed elements \cite{ref:marussig18, ref:buffa19}. 

On the other hand, the use of V-reps (volume representations) \cite{ref:elber16} in IGA is at an early stage and requires careful investigations in both theory and algorithms. A precedent work initiated this effort with a focus on volumetric trimming \cite{ref:antolin19a}, where creating a suitable quadrature mesh for each trimmed volumetric element plays the key role. In this work, we continue along this direction to work on the union operation, where a geometry is built by combining multiple independent patches together. More precisely, given a series of spline patches, we  overlay one on top of another in a certain order, following the same manner as in the multi-mesh finite element method \cite{ref:johansson18, ref:johansson19}. While we employ the Nitsche's method to weakly couple these patches through visible interfaces, the use of one-sided fluxes and minimal stabilization inspired by \cite{ref:buffa19} allows us to prove stability of the proposed approach in completely general geometric configurations.

Note that while the theory is dimension-independent, our algorithms now are only ready for 2D geometries. Creating a suitable quadrature mesh for each interface is a major challenge, especially in 3D. The algorithm for volumes becomes much more involving due to the need of resolving the mesh-to-mesh intersection problem \cite{ref:brivadis15, ref:seitz16, ref:buffa18contact}, and will be the objective of a follow-up work. We also note that our work is related to various methods, such as the isogeometric mortar method \cite{ref:brivadis15}, IGA based on the constructive solid geometry \cite{ref:zuo15, ref:wassermann17}, and the overlapping multi-patch method \cite{ref:kargaran19}. Our method is different from them in handling numerical integration, volume description, and particularly, stabilization.

The paper is organized as follows. Section \ref{sec:notation} sets up the notations and necessary assumptions. The theory of overlapping multi-patch isogeometric analysis is presented in Section \ref{sec:ncmp}. In Section \ref{sec:impl}, we discuss how to create suitable quadrature meshes for interfaces, as well as how to implement the stabilization method. We next show several numerical examples in Section \ref{sec:result} to demonstrate the convergence and conditioning behavior of the proposed method. We conclude the paper and comment on future directions in Section \ref{sec:con}.

\section{Parametrization, mesh and approximation space for domains obtained via union}\label{sec:notation}
For the construction of isogeometric spaces in 1D and in the tensor product case we refer the interested reader to~\cite{ref:buffa19,buffa_acta}. As it is standard in IGA, we denote by $S_p\left(\Xi\right)$ the space of splines of degree $p$ defined on the parameter space $\hat\Omega:=\left(0,1\right)^d$, $d=2,3$, and on the open knot vector $\Xi$.
Let $\Omega_i^\ast \subset\R^d$, $0\le i\le N$ ($N\in\N$), $d\in\{2,3\}$, be a \emph{predomain}, i.e., before the union operation that is described later on. We assume that there exists a bi-Lipschitz map $\mathbf{F}_i\in \left(S_{\mathbf{p}^i}(\mathbf{\Xi}^i) \right)^d$ such that $\Omega_i^{\ast}=\mathbf{F}_i(\hat{\Omega})$, for given degree vector $\mathbf{p}^i$ and knot vector $\mathbf{\Xi}^i$. We define the (physical) \emph{B\'ezier premesh} as the image of the elements in $\hat{\mathcal{M}}_i$ (the parametric B\'ezier mesh naturally induced on $\hat{\Omega}_i$) through $\mathbf{F}_i$:
\begin{equation*}
\mathcal{M}_i^{\ast}:=\{K\subset\Omega_i^{\ast}: K=\mathbf{F}_i(Q), Q\in\hat{\mathcal{M}}_i \}.
\end{equation*}
For each element $K$, we denote as $\widetilde K$ its \emph{support extension} (see \cite{buffa_acta}), i.e. the union of the support of all basis functions that do not vanish on $K$.

%

Let $\hat{V}_{h,i}$ be a refinement of $S_{\mathbf{p}^i}(\mathbf{\Xi}^i)$ and
\begin{equation*}
V_{h,i}=\operatorname{span}\{B_{\mathbf{i},\mathbf{p}}(\mathbf{x}):=\hat{B}_{\mathbf{i},\mathbf{p}}\circ\mathbf{F}_i^{-1}(\mathbf{x}):\mathbf{i}\in\mathbf{I} \},
\end{equation*}
where $\{\hat{B}_{\mathbf{i},\mathbf{p}}: \mathbf{i}\in\mathbf{I} \}$ is a basis of $\hat{V}_{h,i}$. We define a partition $\{\Omega_i \}_{i=0}^N$ of $\Omega$ as 
 \begin{equation*}
 \Omega_i:=\Omega_i^\ast\setminus \bigcup_{l=i+1}^N\Omega_l^\ast,\qquad\forall\ i=0,\dots, N,
 \end{equation*}
 i.e., $\Omega_i$ is the \emph{visible part} of the predomain $\Omega_i^\ast$. We have $\Omega_N=\Omega_N^\ast$. See Figure~\ref{figure1}. Note that this choice of definition of $\Omega_i$ follows~\cite{ref:johansson18} and implies a \emph{hierarchy} of predomains. In particular, if $i>j$, $\Omega_i^\ast$ is \emph{on top of} $\Omega_j^\ast$ in the sense that $\Omega_i^\ast\cap\Omega_j^\ast$ is hidden by $\displaystyle\strut\cup_{k\ge i}\Omega_k$. We define 
  \begin{equation*}
 \Gamma_i:=\partial\Omega_i^\ast\setminus \bigcup_{l=i+1}^N\Omega_l^\ast,\qquad\forall\ i=0,\dots, N,
 \end{equation*}
 i.e., the interface $\Gamma_i$ is the \emph{visible part of the external boundary} of $\Omega_i^\ast$ with outer unit normal $n_i$.
Moreover, we define the \emph{local interfaces}
\begin{equation*}
\Gamma_{ij}:=\Gamma_i\cap\Omega_j,\qquad 0\le j<i\le N, 
\end{equation*}
i.e., $\Gamma_{ij}$ is the subset of the visible boundary of $\Omega_i^\ast$ that intersects $\Omega_j$. We assume that $\Gamma_{ij}$ inherits the orientation of $\Gamma_i$, hence it has unit normal $n_i$, also denoted as $n$ when it is clear from the context to which domain is referred to. Note that $\Gamma_{ij}$ is not connected in general. 

 \begin{figure}[!ht]
	\centering
	\subfloat[][The predomains $\Omega_i^\ast$, $0\le i \le N$.]
	{	\begin{tikzpicture}
	
\definecolor{color1}{rgb}{0,0,1} 
\definecolor{color2}{rgb}{1,0.019607843137255,0.0819607843137255} 
\definecolor{color3}{rgb}{0.019607843137255,0.872549019607843,0.0819607843137255} 

\coordinate (A0) at (0,-1);
\coordinate (B0) at (8,-1);
\coordinate (C0) at (8,3);
\coordinate (D0) at (0,3);

\coordinate (A1) at (4.37,5.01);
\coordinate (B1) at (7.09,3.77);
\coordinate (C1) at (9.32,8.67);
\coordinate (D1) at (6.61,9.91);

\coordinate (A2) at (1.31,4.6);
\coordinate (B2) at  (3.19,5.19);
\coordinate (C2) at  (2.15,8.52);
\coordinate (D2) at (0.27,7.93);

\draw[thick, fill = white] (A0) -- (B0) -- (C0) -- (D0) -- (A0); 

\draw[gray]
\foreach \i in {1, ..., 4} {
	\foreach \j in {1, ..., 4} {
		($(A0)!\i/5!(B0)$) -- ($(D0)!\i/5!(C0)$)
		($(B0)!\i/5!(C0)$) -- ($(A0)!\i/5!(D0)$)
	}
}
; 

\draw[thick,fill=white] (A1) -- (B1) -- (C1) -- (D1) -- (A1); 

\draw[gray]
\foreach \i in {1, ..., 4} {
	\foreach \j in {1, ..., 4} {
		($(A1)!\i/5!(B1)$) -- ($(D1)!\i/5!(C1)$)
		($(B1)!\i/5!(C1)$) -- ($(A1)!\i/5!(D1)$)
	}
}
; 

\draw[thick,fill=white] (A2) -- (B2) -- (C2) -- (D2) -- (A2); 

\draw[gray]
\foreach \i in {1, ..., 4} {
	\foreach \j in {1, ..., 4} {
		($(A2)!\i/5!(B2)$) -- ($(D2)!\i/5!(C2)$)
		($(B2)!\i/5!(C2)$) -- ($(A2)!\i/5!(D2)$)
	}
}
; 
	
	\coordinate (a) at (3.75,1);
	\coordinate (b) at (7,6.75);
	\coordinate (c) at (1.75,6.6);
	
\node at (a)   {\huge$\mathbf{\Omega_0^\ast}$};
\node at (b)   {\huge$\mathbf{\Omega_1^\ast}$};
\node at (c)   {\huge$\mathbf{\Omega_2^\ast}$};

	\coordinate (a) at (2,2);
	\begin{scope}
	\clip (0,0) -- (9.5,0) -- (9.5,10) -- (0,10) -- (0,0);
	\end{scope}

	\end{tikzpicture}
	
	}
	\hspace{0.5cm}
	\subfloat[][$\Omega=\bigcup_{i=0}^N\Omega_i=\bigcup_{i=0}^N\Omega_i^\ast$.]
	{
		\begin{tikzpicture}

\definecolor{color1}{rgb}{0,0,1} 
\definecolor{color2}{rgb}{1,0.019607843137255,0.0819607843137255} 
\definecolor{color3}{rgb}{0.019607843137255,0.872549019607843,0.0819607843137255} 

\coordinate (A0) at (0,-1);
\coordinate (B0) at (8,-1);
\coordinate (C0) at (8,3);
\coordinate (D0) at (0,3);

\coordinate (A1) at (2.91,2.37);
\coordinate (B1) at (5.63,1.13);
\coordinate (C1) at (7.86,6.03);
\coordinate (D1) at (5.15,7.27);

\coordinate (A2) at (2.26,1.68);
\coordinate (B2) at  (4.14,2.27);
\coordinate (C2) at (3.1,5.6);
\coordinate (D2) at (1.22,5.01);

\draw[thick, fill = white] (A0) -- (B0) -- (C0) -- (D0) -- (A0); 

\draw[gray]
\foreach \i in {1, ..., 4} {
	\foreach \j in {1, ..., 4} {
		($(A0)!\i/5!(B0)$) -- ($(D0)!\i/5!(C0)$)
		($(B0)!\i/5!(C0)$) -- ($(A0)!\i/5!(D0)$)
	}
}
; 

\draw[thick,fill=white] (A1) -- (B1) -- (C1) -- (D1) -- (A1); 

\draw[gray]
\foreach \i in {1, ..., 4} {
	\foreach \j in {1, ..., 4} {
		($(A1)!\i/5!(B1)$) -- ($(D1)!\i/5!(C1)$)
		($(B1)!\i/5!(C1)$) -- ($(A1)!\i/5!(D1)$)
	}
}
; 

\draw[thick,fill=white] (A2) -- (B2) -- (C2) -- (D2) -- (A2);  

\draw[gray]
\foreach \i in {1, ..., 4} {
	\foreach \j in {1, ..., 4} {
		($(A2)!\i/5!(B2)$) -- ($(D2)!\i/5!(C2)$)
		($(B2)!\i/5!(C2)$) -- ($(A2)!\i/5!(D2)$)
	}
}
; 

\draw[ultra thick,color2] (1.85,3) -- (2.26,1.68) -- (3.54,2.08); 
\draw[ultra thick,color1] (3.54,2.08) -- (5.63,1.13) -- (6.48,3); 
\draw[ultra thick,color3] (3.54,2.08) --  (4.14,2.27) -- (3.62,3.93); 

	\coordinate (a) at (3.75,1);
\coordinate (b) at (5.5,4.25);
\coordinate (c) at (2.75,3.85);
	\coordinate (d) at (2,1.5);
\coordinate (e) at (6.25,1);
\coordinate (f) at (4.5,3);

\node at (a)   {\huge$\mathbf{\Omega_0}$};
\node at (b)   {\huge$\mathbf{\Omega_1}$};
\node at (c)   {\huge$\mathbf{\Omega_2}$};

\node at (d)   {\textcolor{color2}{\Large$\mathbf{\Gamma_{20}}$}};
\node at (e)   {\textcolor{color1}{\Large$\mathbf{\Gamma_{10}}$}};
\node at (f)   {\textcolor{color3}{\Large$\mathbf{\Gamma_{21}}$}};

\coordinate (a) at (2,2);
\begin{scope}
\clip (0,0) -- (9.5,0) -- (9.5,10) -- (0,10) -- (0,0);
\end{scope}

\end{tikzpicture}
	}
	\caption{Definitions of predomains (a), visible parts of predomains, and local interfaces ($\Gamma_{10}$, $\Gamma_{20}$ and $\Gamma_{21}$) of predomain boundaries (b).}\label{figure1}
\end{figure}
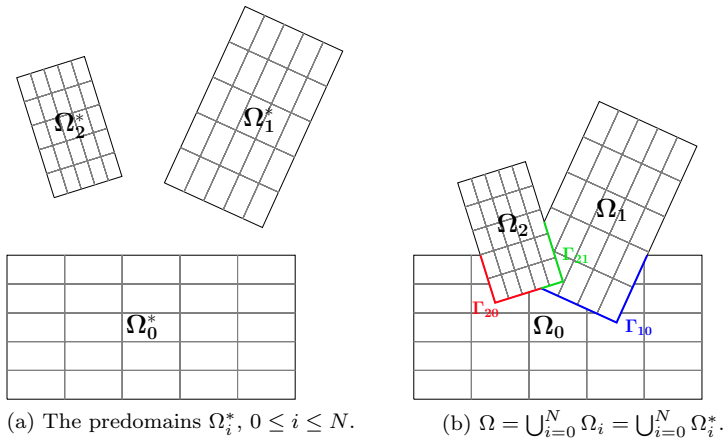

Let $i,j\in \{0,\dots, N\}$, with $i>j$. We define
\begin{equation*}
	\delta_{ij}=
	\begin{cases}
		1 \qquad\text{if}\;\; \Gamma_{ij}\ne\emptyset \\
		0 \qquad \text{otherwise},
	\end{cases} 
	\quad
	N_{\Gamma}:=\sum_{i=1}^N \sum_{j=0}^{i-1} \delta_{ij}.
\end{equation*}
Here we use the convention $\delta_{ii}=1$. $N_{\Gamma}$ counts the number of overlaps $\Gamma_{ij}$ that there exist in the multimesh configuration.
\begin{remark}
The stability estimates derived in this paper degenerate as $N_{\Gamma}$ grows.
\end{remark}	

We refer to $\mathcal M_i :=\{ K\in \mathcal M^\ast_i: K\cap\Omega_i\ne\emptyset  \}$, $\forall\  i=0,\dots,N,$
as the i-th \emph{extended mesh}, consisting of all visible elements of the i-th premesh $\mathcal M^\ast_i$. 
For every $K\in\mathcal M_i$, $0\le i\le N$, let $h_{i,K}:=\operatorname{diam}\left(K\right)$ and $h_{i,\min}:=\min_{K\in\mathcal M_i}h_{i,K}$. We define $\mathsf h_i:\Omega_i\to\mathbb{R}^+$ to be the piecewise constant mesh-size function of $\mathcal M_i$ defined as $\restr{\mathsf h_i}{K}:=h_{i,K}$.


In the following we assume, for the sake of simplicity, that adjacent sub-domains are discretized with similar mesh sizes. 
\begin{assumption}\label{assumption_uniformity}
The meshes locally have compatible sizes in the following sense. There exist $c,C>0$ such that for $\forall\ 1\le j< i \le N: \Gamma_{ij}\ne\emptyset$, $\forall\ K_i\in	\mathcal M_i : \overline K_i\cap\overline\Gamma_{ij}\ne\emptyset$, $\forall\ K_j\in	\mathcal M_j : \overline K_j\cap\overline\Gamma_{ij}\ne\emptyset$ it holds
\begin{equation*}
c  \restr{\mathsf h_j}{K_j}\le \restr{\mathsf h_i}{K_i} \le C \restr{\mathsf h_j}{K_j}.
\end{equation*}		
\end{assumption}	
	
\begin{remark}\label{remarkh}
Assumption~\ref{assumption_uniformity} has a few consequences:
\begin{equation*}
h_{i,K_i}=\left(h_{i,K_i} ^{-1}\right)^{-1}\simeq \left( 2h_{i,K_i}^{-1}\right)^{-1}\simeq \left( h_{i,K_i} ^{-1}+ h_{j,K_j}^{-1}\right)^{-1}.
\end{equation*}
hence  $h_{i,K_i}^{-\frac{1}{2}}\simeq \left( h_{i,K_i} ^{-1}+h_{j,K_j} ^{-1}\right)^{\frac{1}{2}}$.
Similarly, $h_{i,K_i} ^{\frac{1}{2}} \simeq \left(h_{i,K_i}^{-1}+h_{j,K_j} ^{-1}\right)^{-\frac{1}{2}}$.
\end{remark}
In what follows, for each $ 1\le j< i \le N: \Gamma_{ij}\ne\emptyset$ we denote by $\mathsf h_{ij}:\Gamma_{ij}\to\R^{+}$ the piecewise constant functions defined as $\mathsf h_{ij}^{-1}:=\mathsf h_{i}^{-1} + \mathsf h_{j}^{-1}$.
Finally, let us make a mild assumption on the roughness of the interfaces $\Gamma_{ij}$.
\begin{assumption}\label{mesh_assumptions}
There exists $C>0$ such that for $\forall\ 1\le j< i \le N: \Gamma_{ij}\ne\emptyset$, $\forall\ K\in	\mathcal M_j : \overline K\cap\overline\Gamma_{ij}\ne\emptyset$ it holds $\operatorname{meas}_{d-1}(\Gamma_{ij}\cap \overline K)\le C \restr{\mathsf h_j^{d-1}}{K}$.	
\end{assumption}	

\section{Isogeometric Analysis on Non-Conforming Multi-Patches}
\label{sec:ncmp}

\subsection{Model Problem and its Variational Formulation}

Let us consider the Poisson equation as the model problem. The goal is to solve it in the domain $\Omega$ via union,
\begin{equation*}
\Omega = \bigcup_{i=0}^N \Omega_i = \bigcup_{i=0}^N \Omega_i^{\ast}.	
\end{equation*}	
Given $f\in L^2(\Omega)$, $g_D\in H^{\frac{1}{2}}(\Gamma_D)$ and $g_N\in H^{-\frac{1}{2}}(\Gamma_N)$, find $u:\Omega\to\R$ such that
\begin{equation}\label{problem1}
\begin{cases}
-\Delta u = f\qquad&\text{in}\;\Omega\\
u=g_D\qquad&\text{on}\;\Gamma_D\\
\displaystyle{\frac{\partial u}{\partial n}}=g_N\qquad&\text{on}\;\Gamma_N,
\end{cases}
\end{equation}
where $\Gamma_D\cup\Gamma_N=\Gamma=:\partial\Omega$, $\Gamma_D\cap\Gamma_N=\emptyset$. We proceed, similarly to~\cite{ref:becker03, ref:stenberg98}, to rewrite Problem \eqref{problem1} in the following form: find $u\in V$ such that
 \begin{equation}\label{problem2}
 \begin{cases}
 -\Delta u_i = f\qquad&\text{in}\;\Omega_i,\; i=0,\dots, N\\
  u_i - u_j = 0 \qquad&\text{on}\;\Gamma_{ij},\;0\le j<i\le N\\
  \displaystyle{ \pderiv{u_i}{n_i} }+\pderiv{u_j}{n_j} = 0 \qquad&\text{on}\;\Gamma_{ij},\;0\le j<i\le N\\
 u=g_D\qquad&\text{on}\;\Gamma_D\\
 \displaystyle{\frac{\partial u}{\partial n}}=g_N\qquad&\text{on}\;\Gamma_N,
 \end{cases}
 \end{equation}
 where
 \begin{equation*}
 V:=\bigoplus_{i=0}^N\{v_i\in H^1(\Omega_i): \pderiv{v_i}{n_i}\in L^2(\Gamma_i), \restr{v_i}{\partial\Gamma_D\cap \Gamma_i}=0  \}. 
 \end{equation*}
 
 \begin{proposition}\label{prop1}
 	If the solution to Problem \eqref{problem1} satisfies $u\in H^{\frac{3}{2}+\eps}(\Omega)$,  that is,  it is sufficiently regular to have the normal derivative well-defined in $L^2$, then Problems \eqref{problem1} and \eqref{problem2} are equivalent. 
 	
 	
 \end{proposition}
\begin{proof}
We refer the interested reader to~\cite{ref:quarteroni99}.
\end{proof}	

\begin{remark}
Note that the regularity assumption in Proposition~\ref{prop1} can be translated into a regularity assumption of the data and domains following~\cite{ref:becker03}. 
\end{remark}	


We define on each extended B\'ezier mesh $\mathcal M_i$ a conforming approximation space $V_{h,i}$. 
Then, we restrict it to each visible part $\Omega_i$ as
\begin{equation*}
\tilde V_{h,i}=\operatorname{span}\{\restr{B_{\mathbf{i},\mathbf{p}}}{\Omega_i}:\mathbf{i}\in\mathbf{I} \}.
\end{equation*}
The \emph{union isogeometric space} is
\begin{equation*}
V_h:=\bigoplus_{i=0}^N \tilde V_{h,i}.
\end{equation*} 
An element of $V_h$ is a $(N+1)$-tuple $v_h=\left( v_0,\dots,v_N \right)$. In practice, we can treat it as a scalar function thanks to the following embedding:
\begin{equation*}
V_h\hookrightarrow L^2(\Omega),\qquad
v_h(x)\mapsto v_i(x)\qquad\forall\ x\in\Omega_i, \,\forall\ i=0,\dots,N.
\end{equation*} 
 
 In the following we make use of the following notation. The \emph{jump term} on $\Gamma_{ij}$, $i>j$, is defined by $\displaystyle\strut\left[v_h\right]:=\restr{v_i}{\Gamma_{ij}}-\restr{v_j}{\Gamma_{ij}}$, where $\restr{v_i}{\Gamma_{ij}}$ denotes the trace of $v_i$ on $\Gamma_{ij}$. We define the approximation of the normal flux through $\Gamma_{ij}$ as $\displaystyle\strut \langle \pderiv{v_h}{n} \rangle_t:=t\pderiv{v_i}{n_i}+\left(1-t\right)\pderiv{v_j}{n_i}$, $t\in\{\frac{1}{2},1\}$. When $t=\frac{1}{2}$ it is a \emph{symmetric average flux}, while when $t=1$ we choose the \emph{one-sided flux} on the $\Omega_i$ side, in the spirit of~\cite{ref:hansbo03}, i.e. on the side of the domain which is on top of the other. As the choice $t=1$ will turn out to be the most convenient one, we may drop the index $s$ when it is equal to $1$.
 

 
The discrete variational formulation of the problem states as follows: find $u_h\in V_h$ such that
 \begin{equation}\label{non_stabilized_formulation}
 a_h(u_h,v_h) = F_h(v_h)\qquad \forall\ v_h\in V_h\cap H^1_{0,\Gamma_D}(\Omega),
 \end{equation}
 where
 \begin{equation}\label{bilinear_form}
 \begin{aligned}
 a_h(u_h,v_h)&:=\sum_{i=0}^N\int_{\Omega_i}\nabla u_i\cdot \nabla v_i
 -\sum_{i=1}^N\sum_{j=0}^{i-1}\int_{\Gamma_{ij}}\left(\langle \pderiv{u_h}{n}\rangle_t\left[v_h\right]+\left[u_h\right]\langle \pderiv{v_h}{n}\rangle_t\right)\\
 &+\beta\sum_{i=1}^N\sum_{j=0}^{i-1}\int_{\Gamma_{ij}}\mathsf h_{ij}^{-1}\left[u_h\right]\left[v_h\right],
 \end{aligned}
 \end{equation}
 with $t\in\{\frac{1}{2},1\}$, and
 \begin{equation}
 \begin{aligned}
 F_h(v_h):=&\sum_{i=0}^N\int_{\Omega_i}fv_i + \int_{\Gamma_N} g_N v_h.
 \end{aligned}
 \end{equation}
 Note that $\beta>0$ is a penalty parameter related to the spline degree, and its specific choice will be discussed in Section \ref{sec:result}.
 
 \begin{proposition}
The discrete variational formulation in Equation~\eqref{non_stabilized_formulation} is consistent, i.e., the solution $u$ to the strong Problem~\eqref{problem1} satisfies Equation~\eqref{non_stabilized_formulation} as well.
 \end{proposition}
\begin{proof}
	The proof is quite classical. See for instance~\cite{ref:becker03,ref:hansbo03}.
\end{proof}	
 
\subsection{Quasi-interpolation strategy}
 Let us generalize the interpolation strategy employed in~\cite{ref:buffa19}. Giving a Sobolev function living in the whole physical domain $\Omega$, we consider its components in each visible part $\Omega_i$, extend them to the predomains $\Omega_i^\ast$ in order to be able to interpolate on each mesh $\mathcal M_i$ and finally glue together the interpolated functions. 
 
 We construct a spline quasi-interpolant operator for the space $\tilde V_h$. Given a fixed $s> \frac{1}{2}$, for every $i\in\{0,\dots,N\}$, we have
 \begin{equation*}
 	\tilde \Pi^i_h:H^{s+1}(\Omega_i)\to \tilde V_{h,i},\qquad
 	u\mapsto \Pi_h^i\left(\restr{E^i\left(u \right)}{\Omega_i^\ast}\right) ,
 \end{equation*}
 where $E^i:H^{r+1}(\Omega_i)\to H^{s+1}(\R^d)$
 is the Sobolev-Stein extension operator (see Section 3.2 in~\cite{oswald}) and $\Pi_h^i: H^{s+1}(\Omega_i^{\ast})\to V_h$ is a standard spline quasi-interpolant operator (see~\cite{Buffa2016}). Then, we glue together the local operators as follows:
 \begin{equation*}
 	\Pi_h:H^{s+1}(\Omega)\to V_h,\qquad 
 	v\mapsto \bigoplus_{i=0}^N \tilde \Pi_h^i\left(v_i\right)\in V_h.
 \end{equation*}
 
 \begin{proposition}[Interpolation error estimate]
 	Let $\frac{1}{2} < s \le p$. Then there exists $C>0$ such that	
 	\begin{equation*}
 		\norm{u-\Pi_h u}_{1,h}\le C h^s\norm{u}_{H^{s+1}(\Omega)}\qquad\forall\ u\in H^{s+1}(\Omega).
 	\end{equation*}	
 \end{proposition}	
 \begin{proof}
 	The proof is rather standard and we omit it.
 \end{proof}

 \subsection{Minimal Stabilization}
 \label{sec:ministab}
 
 It has been shown in~\cite{ref:buffa19} that Problem~\eqref{non_stabilized_formulation} may suffer of instability due to the evaluation of the normal derivatives in bad cut elements\footnote{Roughly speaking, elements for which a large portion has been cut away. The exact definition will be given below.}. In the two-patch situation, such as in Figure~\ref{interface_overlap}(a), we do not have the instability issue as soon as we are using the one-sided flux from top elements that are not cut. However, we do have this issue in general cases with many patches; see Figure~\ref{interface_overlap}(b), where the one-sided flux regarding the interface $\Gamma_{ij}$ may come from the red element, a, possibly bad, cut element. In this regard, the so-called minimal stabilization in~\cite{ref:buffa19} comes to help. The minimal stabilization mainly follows three steps: (1) distinguishing good and bad elements depending on how elements are cut; (2) finding a good neighbor for each bad element; and (3) stabilizing normal derivatives with the help of the good neighbor. In what follows, we need to accommodate this method in the context of multi-patches that overlap.
 
 \begin{figure}[!ht]
 	\centering
 	\subfloat[][]
 	{\begin{tikzpicture}

\definecolor{color1}{rgb}{0,0,1} 
\definecolor{color2}{rgb}{1,0.019607843137255,0.0819607843137255} 
\definecolor{color3}{rgb}{0.019607843137255,0.872549019607843,0.0819607843137255} 
\draw[ultra thick, fill = white] (0,-1) -- (8,-1) -- (8,3) -- (0,3) -- (0,-1); 
\draw[step=1cm,lightgray,thin] (0,-1) grid (8,3); 

\coordinate (A) at (2.74,2.62);
\coordinate (B) at (5.46,1.38);
\coordinate (C) at (7.69,6.28);
\coordinate (D) at (4.98,7.52);

\coordinate (E) at (5.74,2);
\coordinate (F) at (6,2);
\coordinate (G) at (5.99,2.55);

\coordinate (H) at (2.66,2.1);
\coordinate (I) at (4.54,2.7);
\coordinate (L) at (3.5,6.03);
\coordinate (M) at (1.62,5.44);


\draw[ultra thick, fill = white] (A)-- (B) -- (C) -- (D) -- (A); 

\draw[lightgray]
\foreach \i in {1, ..., 4} {
	\foreach \j in {1, ..., 4} {
		($(A)!\i/5!(B)$) -- ($(D)!\i/5!(C)$)
		($(B)!\i/5!(C)$) -- ($(A)!\i/5!(D)$)
	}
}
; 



\coordinate (N) at (3.38,2.33);
\coordinate (O) at (3.85,2.12);
\coordinate (P) at (4.04,2.54);


\draw[ultra thick, blue] (2.92,3)--(A) --(N)-- (B)--(6.2,3) ; 

\coordinate (a) at (4,1);
\coordinate (b) at (5.3,4.25);
\coordinate (c) at (3,4.25);
\coordinate (d) at (5.99,1.2);

\node at (a)   {\huge$\mathbf{\Omega_j}$};
\node at (b)   {\huge$\mathbf{\Omega_i}$};
\node at (d)   {\textcolor{blue}{\Large$\mathbf{\Gamma_{ij}}$}};

\begin{scope}
\clip (0,0) -- (8,0) -- (8,7.5) -- (0,7.5) -- (0,0);
\end{scope}

\end{tikzpicture}
 	}
 	\hspace{0.5cm}
 	\subfloat[][]
 	{
 		\begin{tikzpicture}

\definecolor{color1}{rgb}{0,0,1} 
\definecolor{color2}{rgb}{1,0.019607843137255,0.0819607843137255} 
\definecolor{color3}{rgb}{0.019607843137255,0.872549019607843,0.0819607843137255} 
\draw[ultra thick, fill = white] (0,-1) -- (8,-1) -- (8,3) -- (0,3) -- (0,-1); 
\draw[step=1cm,lightgray,thin] (0,-1) grid (8,3); 

\coordinate (A) at (2.74,2.62);
\coordinate (B) at (5.46,1.38);
\coordinate (C) at (7.69,6.28);
\coordinate (D) at (4.98,7.52);

\coordinate (E) at (5.74,2);
\coordinate (F) at (6,2);
\coordinate (G) at (5.99,2.55);

\coordinate (H) at (2.66,2.1);
\coordinate (I) at (4.54,2.7);
\coordinate (L) at (3.5,6.03);
\coordinate (M) at (1.62,5.44);


\draw[ultra thick, fill = white] (A)-- (B) -- (C) -- (D) -- (A); 

\draw[lightgray]
\foreach \i in {1, ..., 4} {
	\foreach \j in {1, ..., 4} {
		($(A)!\i/5!(B)$) -- ($(D)!\i/5!(C)$)
		($(B)!\i/5!(C)$) -- ($(A)!\i/5!(D)$)
	}
}
; 

\draw[ultra thick, fill = white] (H)-- (I) -- (L) -- (M) -- (H); 

\draw[lightgray]
\foreach \m in {1, ..., 4} {
	\foreach \n in {1, ..., 4} {
		($(H)!\m/5!(I)$) -- ($(M)!\m/5!(L)$)
		($(I)!\m/5!(L)$) -- ($(H)!\m/5!(M)$)
	}
}
;

\coordinate (N) at (3.38,2.33);
\coordinate (O) at (3.85,2.12);
\coordinate (P) at (4.04,2.54);

\draw[ fill =red] (N)-- (O) -- (P) -- (N); 

\draw[ultra thick, blue] (N)-- (B) -- (6.2,3); 

\coordinate (a) at (4,1);
\coordinate (b) at (5.3,4.25);
\coordinate (c) at (3,4.25);
\coordinate (d) at (5.99,1.2);

\node at (a)   {\huge$\mathbf{\Omega_j}$};
\node at (b)   {\huge$\mathbf{\Omega_i}$};
\node at (c)   {\huge$\mathbf{\Omega_k}$};
\node at (d)   {\textcolor{blue}{\Large$\mathbf{\Gamma_{ij}}$}};

\begin{scope}
\clip (0,0) -- (8,0) -- (8,7.5) -- (0,7.5) -- (0,0);
\end{scope}

\end{tikzpicture}
 	}
	\caption{Two-patch and three-patch overlapping along the interface $\Gamma_{ij}$.}
 	 \label{interface_overlap}
 \end{figure}
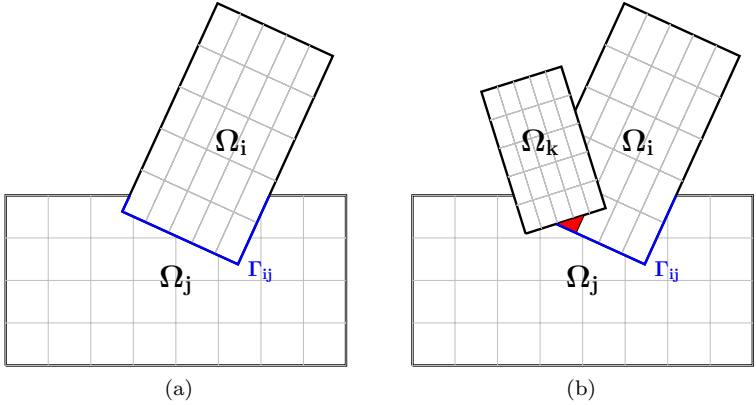
 
First, for each extended B\'{e}zier mesh $\hat {\mathcal M}_i$ ($i=0,\dots, N-1$), we partition its elements into two disjoint sub-families in the following. 

 \begin{definition}\label{def_goodandbad}
 	Fix $i=0,\dots, N-1$, $\theta\in (0,1]$ and let $Q\in\hat{\mathcal{M}}_i$. We say that $Q$ is a \emph{good element} if
 	\begin{equation}\label{good_bad_def}
 	\frac{\operatorname{meas}_{d}\left(\hat\Omega\cap Q\right)}{\operatorname{meas}_d\left(Q\right)}\ge\theta.
 	\end{equation}
 	Otherwise, $Q$ is a \emph{bad element}. As $\F_i$ is assumed to be bi-Lipschitz, this classification on parametric elements induces naturally a classification for the corresponding physical elements.
 	
	We denote as $\mathcal M^g_i$ and $\mathcal M^b_i$ the collection of good and bad physical B\'ezier elements, respectively. Note that all the non-cut elements in $\mathcal M_i$ are good elements, and all the bad elements are cut elements.
 \end{definition}
 
 \begin{definition}\label{def_neigh}
 	Given a physical B\'{e}zier element $K\in\mathcal M_i$, the set of its neighbors is defined as
 	\begin{equation*}
 	\mathcal N (K):=\{K'\in\mathcal M_k : \operatorname{dist}\left(K,K'\right)\le C\restr{\mathsf h_i}{K},\quad k=0,1,\ldots,N \},
 	\end{equation*}
 	where $C>0$ does not depend on the mesh-sizes.
 \end{definition}

 	Next, for each bad cut element $K\in\mathcal M_i^b$, $0\le i < N$,  we associate a good neighbor $K'$ (a neighbor that is a good element). Note that in principle we allow $K'\in \mathcal M_k$ with $i\neq k$, i.e., a good neighbor can belong to the mesh of another domain. Let $K\in\mathcal M_i^b$, $0\le i < N$. Its associated good neighbor $K'$ is chosen according to the following the procedure:
 	\begin{itemize}
 		\item \emph{Step 1}: We search in $\mathcal M_i^g$ to check if there is any good neighbor that belongs to the same mesh. If so, we have $K'\in\mathcal N(K)\cap\mathcal M_i^g$; otherwise we proceed to \emph{Step 2}.
 		\item \emph{Step 2}: We switch to a top domain and search for $K'\in\mathcal N(K)\cap\mathcal M_k^g$ ($k>i$).
 	\end{itemize}
 Figure~\ref{fig:goodnb} shows the good neighbor in the same and different domains.

 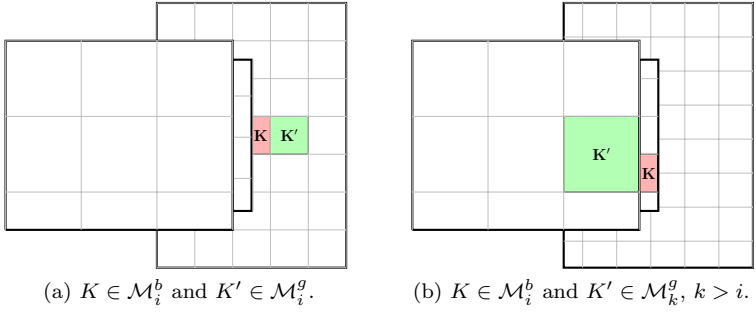
\begin{figure}
 	\centering
 	\subfloat[][$K\in\mathcal M_i^b$ and $K'\in\mathcal M_i^g$.]
 	{	\begin{tikzpicture}
	
	\definecolor{color2}{rgb}{0.872549019607843,0.019607843137255,0.0819607843137255}
	
	\draw[ultra thick] (4,8) -- (4,9) -- (9,9) -- (9,2) -- (4,2) -- (4,3); 
	\filldraw[fill=red!30!white, draw=black] (6.5,5) rectangle (7,6); 
    \filldraw[fill=green!30!white, draw=black] (7,5) rectangle (8,6); 
	\draw[step=1cm,lightgray,very thin] (4,2) grid (9,9); 
	\draw[ultra thick,fill=white] (6,3.5) -- (6.5,3.5) -- (6.5,7.5) -- (6,7.5) -- (6,3.5); 
	\draw[step=1.09cm,lightgray,very thin] (6,3.5) grid (6.5,7.5); 
	\draw[ultra thick,fill=white] (0,3) -- (6,3) -- (6,8) -- (0,8) -- (0,3); 
	\draw[step=2cm,lightgray,very thin] (0,3) grid (6,8); 
	
	\coordinate (a) at (6.75,5.5);
	\coordinate (b) at (7.5,5.5);
	\coordinate (d) at (2.5,7.5);
	\coordinate (c) at (2,4);
	
	\node at (a)   {\large$\mathbf{K}$};
	\node at (b)   {\large$\mathbf{K'}$};
	
	\coordinate (a) at (2,2);
\begin{scope}
\clip (1,2) -- (9,2) -- (9,9) -- (0,9) -- (1,2);
\end{scope}

	\end{tikzpicture}
	
 	}
 	\hspace{0.5cm}
 	\subfloat[][$K\in\mathcal M_i^b$ and $K'\in\mathcal M_k^g$, $k>i$.]
 	{
 		\begin{tikzpicture}
	
\definecolor{color2}{rgb}{0.872549019607843,0.019607843137255,0.0819607843137255}

\draw[ultra thick] (4,8) -- (4,9) -- (9,9) -- (9,2) -- (4,2) -- (4,3); 
\draw[step=0.9cm,lightgray,very thin] (4,2) grid (9,9); 
\draw[ultra thick,fill=white] (6,3.5) -- (6.5,3.5) -- (6.5,7.5) -- (6,7.5) -- (6,3.5); 
\filldraw[fill=red!30!white, draw=black] (6,4) rectangle (6.48,5); 
\draw[step=1cm,lightgray,very thin] (6,3.5) grid (6.5,7.5); 
\draw[ultra thick,fill=white] (0,3) -- (6,3) -- (6,8) -- (0,8) -- (0,3); 
\filldraw[fill=green!30!white, draw=black] (4,4) rectangle (5.98,6); 
\draw[step=2cm,lightgray,very thin] (0,3) grid (6,8); 

\coordinate (a) at (5,5);
\coordinate (b) at (6.25,4.5);

\node at (a)   {\large$\mathbf{K'}$};
\node at (b)   {\large$\mathbf{K}$};

\begin{scope}
  \clip (1,2) -- (9,2) -- (9,9) -- (0,9) -- (1,2);
\end{scope}

\end{tikzpicture}
 	}
 	\caption{The good neighbor $K'$ located in the same domain as $K$ (a), and in a different (top) domain (b).}
 	 \label{fig:goodnb}
 \end{figure}
 Let us denote for each $0\le j<N$
 \begin{equation*}
 	\partial\Omega_j:= \overline \Gamma \cup \bigcup_{i>j}\overline \Gamma_{ij}.
 \end{equation*}	
 We define the stabilization operator
 \begin{equation*}
 \begin{aligned}	
  R_h:V_h&\to L^2\left(\bigcup_{j=0}^{N-1}\partial\Omega_j\right)\\  
  v_h&\mapsto\left(R_0(v_h),\dots,R_{N-1}(v_h), \frac{\partial v_N}{\partial n_N}\right).
  \end{aligned}
\end{equation*}
For $0\le j<N$, $R_j:\widetilde V_{h,j}\to L^2\left(\partial\Omega_j \right)$ is defined locally $\forall\ K\in\mathcal M_j :  \overline K\cap\partial\Omega_j\ne\emptyset$:
\begin{itemize} 
\item if $\overline K \cap\overline \Gamma_j\ne\emptyset$ or $K\in\mathcal M_j^g$, then
 	\begin{equation*}
	\restr{R_j(v_h)}{\overline K\cap\partial\Omega_j}:=\restr{\pderiv{v_j}{n_j}}{\overline K\cap\partial\Omega_j},
\end{equation*}	
\item if $K\in\mathcal M_j^b$, $K'\in\mathcal M_i^g$, $0\le j<i\le N$ its good neighbor, then
 	\begin{equation}
	\restr{R_j(v_h)}{\overline K\cap\partial\Omega_j}:=\restr{\pderiv{\mathcal E\left( P_{K'}\left( \restr{v_i}{K'}\right)\right)}{n_j}}{\overline K\cap\partial\Omega_j},
\label{eq:nmdevstab}
\end{equation}	
where $P_{K'}:L^2(K')\to\mathbb{Q}_p(K')$ is the $L^2$-orthogonal projection onto $\mathbb{Q}_p(K')$ (i.e., the space of polynomials on $K'$ of degree $p$ in each coordinate direction) and $\mathcal E:\mathbb{Q}_p(K')\to\mathbb{Q}_p(K'\cup K)$ is the polynomial natural extension.
\end{itemize}


 From~\cite{ref:buffa19}, we know that the following properties hold for the patchwise stabilization operators. For $0\le j<i\le N$ we have:
\begin{itemize}
	\item a \emph{stability property}: $\forall\ v_h\in V_h$, $\forall\ K\in\mathcal M_j^b$, $0\le j<i\le N$, let $K'$ be its associated good neighbor. Then:
	\begin{equation}\label{ineq_trace}
	\norm{\mathsf h_l^{\frac{1}{2}}R_l(v_h)}_{L^2(\Gamma_{ij}\cap K)}\le  C_S \norm{\nabla v_h}_{L^2(K'\cap\Omega_i)},\qquad l\in\{i,j\},
	\end{equation}
	\item an \emph{approximation property}: $\forall\ v\in H^{s+1}\left(\Omega\right)$, $s> \frac{1}{2}$,
	\begin{equation}\label{ineq_approx}
		\norm{\mathsf h_l^{\frac{1}{2}}\left(R_l(\Pi_h(v))-\frac{\partial v}{\partial n}\right)}_{L^2(\Gamma_{ij}\cap \overline K)}\le C_A h^s \norm{v}_{H^{s+1}(\tilde K\cup \tilde K')},\qquad l\in\{i,j\}.
	\end{equation}
\end{itemize}

Let us denote, for $0\le j\le i\le N$ such that $\Gamma_{ij}\ne\emptyset$ and $t\in\{\frac{1}{2},1\}$,
 \begin{equation}\label{global_stab_op}
 \langle R_h\left(v_h\right)\rangle_t : = t \restr{R_i\left(v_h\right)}{\Gamma_{ij}}+\restr{\left(1-t\right) R_j\left(v_h\right)}{\Gamma_{ij}}.
 \end{equation}
 
\begin{remark}
Although $K$ and $K'$ may belong to different meshes, the proofs of~\eqref{ineq_trace} and~\eqref{ineq_approx} immediately follow those in~\cite{ref:buffa19} due to Assumption~\ref{assumption_uniformity}.
\end{remark}	
	
 We propose the following stabilized weak formulation: find  $u_h\in V_h\cap H^1_{0,\Gamma_D}(\Omega)$ such that
\begin{equation}\label{stabilized_formulation}
\overline{a}_h(u_h,v_h) = F_h(v_h)\qquad \forall\ v_h\in V_h\cap H^1_{0,\Gamma_D}(\Omega).
\end{equation}
Here, the bilinear form is defined, for $t\in\{\frac{1}{2},1\}$, as
\begin{equation}\label{stab_bilinear_form}
\begin{aligned}
\overline{a}_h(u_h,v_h):=\sum_{i=0}^N&\int_{\Omega_i}\nabla u_i\cdot \nabla v_i-\sum_{i=1}^N\sum_{j=0}^{i-1}\int_{\Gamma_{ij}}\left(\langle R_h( u_h)\rangle_t\left[v_h\right]+\left[u_h\right]\langle R_h(v_h)\rangle_t\right)\\
&+\beta\sum_{i=1}^N\sum_{j=0}^{i-1}\int_{\Gamma_{ij}}\mathsf h_{ij}^{-1}\left[u_h\right]\left[v_h\right],
\end{aligned}
\end{equation}
where $\beta>0$ is a penalty parameter. We will employ the following mesh-dependent norm for the subsequent analysis,
\begin{equation}\label{norm}
	\norm{u_h}^2_{1,h}:=\sum_{i=0}^N\norm{\nabla u_i}^2_{L^2(\Omega_i)}+\sum^N_{i=1}\sum_{j=0}^{i-1}\norm{\mathsf h_{ij}^{-\frac{1}{2}}\left[ u_h \right]}^2_{L^2(\Gamma_{ij})}.
\end{equation}

\begin{remark}
	Note that, by the choice of the hierarchy of the predomains made in Section~\ref{sec:notation} and the definition of the global stabilization operator~\eqref{global_stab_op}, the one-sided flux approximation is better than the symmetric average. It allows indeed to modify the weak formulation much less frequently.
\end{remark}

Our goal is to show that Problem~\eqref{stabilized_formulation} is well-posed in the sense of the following definition.

\begin{definition}\label{def_stability}
Problem~\eqref{stabilized_formulation} is stable if there exists $\overline\beta>0$, independent on how the domains overlap, such that for every $\beta\ge\overline\beta$ the bilinear form $\overline a_h(\cdot,\cdot)$ is bounded and coercive w.r.t. $\norm{\cdot}_{1,h}$.
\end{definition}	

\begin{theorem}
Problem~\eqref{stabilized_formulation} is stable in the sense of Definition~\ref{def_stability}.
\end{theorem}	

\begin{proof}
Let us start with continuity:
\begin{equation*}	
\begin{aligned}
&\abs{\overline{a}_h(u_h,v_h)}\le \underbrace{ \sum_{i=0}^N\norm{\nabla u_i}_{L^2(\Omega_i)}\norm{\nabla v_i}_{L^2(\Omega_i)}}_{\RN{1}}\\
&+\sum^N_{i=1}\sum_{j=0}^{i-1}\underbrace{\Big(\norm{\mathsf h_i^{\frac{1}{2}}\langle R_h(u_h)\rangle_t}_{L^2(\Gamma_{ij})}\norm{\mathsf h_{ij}^{-\frac{1}{2}}\left[ v_h\right]}_{L^2(\Gamma_{ij})}}_{\RN{2}}\\
& +\underbrace{\norm{\mathsf h_{ij}^{-\frac{1}{2}}\left[ u_h\right]}_{L^2(\Gamma_{ij})}\norm{\mathsf {h}_i^{\frac{1}{2}}\langle R_h(v_h)\rangle_t}_{L^2(\Gamma_{ij})}\Big)}_{\RN{3}}\\
& +\underbrace{ \beta\sum_{i=1}^N\sum_{j=0}^{i-1}\norm{\mathsf h_{ij}^{-\frac{1}{2}}\left[u_h\right]}_{L^2(\Gamma_{ij})}\norm{\mathsf h_{ij}^{-\frac{1}{2}}\left[v_h\right]}_{L^2(\Gamma_{ij})}}_{\RN{4}}.
\end{aligned}
\end{equation*}
In order to obtain $\RN{2},\RN{3}$ we have multiplied and divided by $\mathsf h_i^{\frac{1}{2}}$ and used Remark~\ref{remarkh}.

It is straightforward to bound $\RN{1}$ and $\RN{4}$. We focus on $\RN{2}$, whereas bounding $\RN{3}$ is analogous.
Using Assumption~\ref{assumption_uniformity} and the definition of the norm~\eqref{norm}, we have
\begin{equation}\label{pass1}
\RN{2}\le  t \sum_{i=1}^N\sum_{j=0}^{i-1}\norm{\mathsf h_i^{\frac{1}{2}} R_i(u_h)}_{L^2(\Gamma_{ij})} \norm{v_h}_{1,h}+C(1-t)\sum_{i=1}^N\sum_{j=0}^{i-1}\norm{\mathsf h_j^{\frac{1}{2}} R_j(u_h)}_{L^2(\Gamma_{ij})} \norm{v_h}_{1,h}.
\end{equation}
Note that
\begin{equation}\label{pass2}
\begin{aligned}
\sum_{i=1}^N\sum_{j=0}^{i-1}&\norm{\mathsf{h}_i^{\frac{1}{2}}R_i(u_h)}_{L^2(\Gamma_{ij})}\norm{v_h}_{1,h}
\le C\sum_{i=1}^N\sum_{j=0}^{i-1}\delta_{ij}\norm{\nabla u_h}_{L^2(\Omega_i)}\norm{v_h}_{1,h}\\
=&  C\sum_{i=1}^N\norm{\nabla u_h}_{L^2(\Omega_i)} \sum_{j=0}^{i-1}\delta_{ij} \norm{v_h}_{1,h}
\le C \max_{1\le i\le N}\sum_{j=0}^{i-1}\delta_{ij}\sum_{i=1}^N\norm{\nabla u_h}_{L^2(\Omega_i)}\norm{v_h}_{1,h}\\
\le&  C \sum_{i=1}^N\sum_{j=0}^{i-1}\delta_{ij}\sum_{i=1}^N\norm{\nabla u_h}_{L^2(\Omega_i)}\norm{v_h}_{1,h}
\le  CN_{\Gamma}\sum_{i=1}^N\norm{\nabla u_i}_{L^2(\Omega_i)}\norm{v_h}_{1,h},
\end{aligned}
\end{equation}
where $C$ depends on the constant of the stability property~\eqref{ineq_trace}. In a similar fashion it is possible to prove
\begin{equation}\label{pass3}
\sum_{i=1}^N\sum_{j=0}^{i-1}\norm{\mathsf{h}_j^{\frac{1}{2}}R_j(u_h)}_{L^2(\Gamma_{ij})}\norm{v_h}_{1,h}\le CN_{\Gamma}\sum_{i=0}^N\norm{\nabla u_i}_{L^2(\Omega_i)}\norm{v_h}_{1,h}.
\end{equation}
 We then put together~\eqref{pass1}, \eqref{pass2} and~\eqref{pass3} to obtain
\begin{equation}\label{eqcontbd}
\begin{aligned}
\RN{2}
\le C N_{\Gamma}\sum_{i=0}^N \norm{\nabla u_i}_{L^2(\Omega_i)}\norm{v_h}_{1,h}.
	\end{aligned}
\end{equation}
Hence, it holds $\abs{\overline{a}_h(u_h,v_h)}\le C\norm{u_h}_{1,h}\norm{v_h}_{1,h}$, where $C=\mathcal O (N_{\Gamma})$.

Now, let us move to coercivity,
\begin{equation}
\begin{aligned}
\overline{a}_h&(u_h,u_h) = \sum_{i=0}^N\norm{\nabla u_i}^2_{L^2(\Omega_i)}-2\sum_{i=1}^N\sum_{j=0}^{i-1}\int_{\Gamma_{ij}}\langle R_h(u_h)\rangle_t \left[ u_h \right] \\
&+ \beta \sum_{i=1}^N\sum_{j=0}^{i-1} \norm{\mathsf h_{ij}^{-1}\left[u_h \right]}^2_{L^2(\Gamma_{ij})}\\
\ge& \sum_{i=0}^N\norm{\nabla u_i}^2_{L^2(\Omega_i)} -\frac{1}{\alpha}\sum_{i=1}^N\sum_{j=0}^{i-1}\norm{\mathsf h_i^{\frac{1}{2}}\langle R_h(u_h)\rangle_t}^2_{L^2(\Gamma_{ij})}\\
&-\alpha\sum_{i=1}^N\sum_{j=0}^{i-1}\norm{\mathsf h_{ij}^{-\frac{1}{2}}\left[u_h\right]}^2_{L^2(\Gamma_{ij})}+ \beta \sum_{i=1}^N\sum_{j=0}^{i-1} \norm{\mathsf h_{ij}^{-1}\left[u_h \right]}^2_{L^2(\Gamma_{ij})}\\
\ge&\left( 1-\frac{C}{\alpha} \right)\sum_{i=0}^N\norm{\nabla u_i}^2_{L^2(\Omega_i)} +\left( \beta-\alpha \right)\sum_{i=1}^N\sum_{j=0}^{i-1} \norm{\mathsf h_{ij}^{-1}\left[u_h \right]}^2_{L^2(\Gamma_{ij})},
	\end{aligned}
\end{equation}
from which we deduce coercivity provided that $C<\alpha<\beta$. Again $C=\mathcal O (N_{\Gamma})$.
\end{proof}

\begin{theorem}[A priori error estimate]
Let $\frac{1}{2}< s\le p$, $u\in H^{s+1}(\Omega)$ be the solution of the continuous problem~\eqref{problem1} and $u_h\in V_h$ the solution of~\eqref{stabilized_formulation}. Then, there exists $C>0$, depending on $N_{\Gamma}$, such that
\begin{equation*}
\norm{u-u_h}_{1,h}\le C h^s \norm{u}_{H^{s+1}(\Omega)}.	
\end{equation*}	
\end{theorem}

\begin{proof}
Let $u$ be the solution of the strong problem~\eqref{problem1} and $a_h(\cdot,\cdot)$ the bilinear form defined in~\eqref{bilinear_form}. Let us start with the triangular inequality and coercivity of $\overline{a}_h(\cdot,\cdot)$:
\begin{equation}\label{first_passage}
\begin{aligned}
\norm{u-u_h}_{1,h}\le&\norm{u-v_h}_{1,h}+\norm{v_h-u_h}_{1,h}\\
\le&\norm{u-v_h}_{1,h}+\alpha\sup_{\substack{w_h\in V_h\\ w_h\ne 0}}\frac{\overline{a}_h(v_h-u_h,w_h)}{\norm{w_h}_{1,h}},
\end{aligned}
\end{equation}
where $\alpha>0$ is the coercivity constant. Recalling that $u$ solves~\eqref{non_stabilized_formulation} and $u_h$ solves~\eqref{stabilized_formulation}. After properly rearranging the terms, we have
\begin{equation}
\begin{aligned}
\overline{a}_h&(v_h-u_h,w_h)=\overline{a}_h(v_h,w_h)-\overline{a}_h(u_h,w_h)
 =\overline{a}_h(v_h,w_h)-a_h(u,w_h)\\
=&\sum_{i=0}^N \int_{\Omega_i}\nabla(v_i-u_i)\cdot\nabla w_i \\
&- \sum_{i=1}^N\sum_{j=0}^{i-1}\int_{\Gamma_{ij}}\left( \langle R_h(v_h) - \pderiv{u}{n_i} \rangle_t \left[ w_h \right] + \langle R_h(w_h) \rangle_t \left[ v_h \right] -\langle \pderiv{w_h}{n_i} \rangle_t \left[ u\right]  \right) \\
&+ \beta \sum_{i=1}^N\sum_{j=0}^{i-1}\int_{\Gamma_{ij}}\mathsf h_{ij}^{-1}\left[ v_h-u\right]\left[ w_h\right] .
\end{aligned}
\end{equation}
Reminding that $\left[u\right]=0$, we have
\begin{equation*}
\int_{\Gamma_{ij}}\langle R_h(w_h) \rangle_t \left[ v_h \right]
=\int_{\Gamma_{ij}}\langle R_h(w_h)\rangle_t \left[v_h-u\right].
\end{equation*}
Hence
\begin{equation}
\begin{aligned}
\overline{a}_h(v_h-u_h,w_h)=&\underbrace{\sum_{i=0}^N \int_{\Omega_i}\nabla(v_i-u_i)\cdot\nabla w_i}_{\RN{1}} \\
&-\underbrace{\sum_{i=1}^N\sum_{j=0}^{i-1}\int_{\Gamma_{ij}}\left( \langle R_h(v_h) - \pderiv{u}{n_i} \rangle_t \left[ w_h \right] + \langle R_h(w_h) \rangle_t \left[ u- v_h \right] \right)}_{\RN{2}\quad \text{\&}\quad\RN{3}}  \\
&+ \underbrace{\beta \sum_{i=1}^N\sum_{j=0}^{i-1}\int_{\Gamma_{ij}}\mathsf h_{ij}^{-1}\left[ v_h-u\right]\left[ w_h\right]}_{\RN{4}} 
\end{aligned}
\end{equation}
First of all, we note that 
\begin{equation}
\RN{1} + \RN{4} \le C \norm{v_h-u}_{1,h}\norm{w_h}_{1,h}.
\end{equation}
Then, in order to bound $\RN{3}$, we multiply and divide by $\mathsf h_{ij}^{-\frac{1}{2}}$ and use Remark~\ref{remarkh} and inequality~\eqref{eqcontbd}.
\begin{equation}
\begin{aligned}
\RN{3} \le & \sum_{i=1}^N\sum_{j=0}^{i-1}\norm{\mathsf h_i^{\frac{1}{2}} \langle R_h(w_h)\rangle_t}_{L^2(\Gamma_{ij})}   \norm{ \mathsf h_{ij}^{-\frac{1}{2}}\left[ u-v_h\right]}_{L^2(\Gamma_{ij})}\\
\le& C N_{\Gamma}  \left(\sum_{i=0}^N\norm{\nabla w_h}_{L^2(\Omega_i)}\right) \norm{u-v_h}_{1,h}\le C  N_{\Gamma} \norm{w_h}_{1,h}\norm{u-v_h}_{1,h}
\end{aligned}
\end{equation}
Similarly, we have
\begin{equation}
\begin{aligned}
\RN{2} \le &\sum_{i=1}^N\sum_{j=0}^{i-1} \norm{\mathsf h_i^{\frac{1}{2}} \langle R_h (v_h)-\pderiv {u}{n_i} \rangle_t}_{L^2(\Gamma_{ij})} \norm{\mathsf h_{ij}^{-\frac{1}{2}}\left[w_h \right]}_{L^2(\Gamma_{ij})}\\
\le &\norm{\mathsf h_i^{\frac{1}{2}}\left(R_h(v_h)-\pderiv{u}{n_i}\right)}_{L^2(\Gamma_{ij})}\norm{w_h}_{1,h} .
\end{aligned}
\end{equation}
Finally, we choose $v_h=\Pi (u)$, and by the approximation properties of $\Pi_h$ and $R_h$, namely~\eqref{ineq_approx}, we obtain
\begin{equation*}
\RN{1} + \RN{2}  + \RN{3} + \RN{4}\le C h^s\norm{u}_{H^{s+1}\left(\Omega\right)}\norm{w_h}_{1,h}.
\end{equation*}
\end{proof}

\section{Implementation Aspects of the Union Operation}
\label{sec:impl}

In this section, we discuss the implementation of the union operation. Given two patches $\Omega_0^{\ast}$ and $\Omega_1^{\ast}$, recall that we create their union by first trimming $\Omega_0^{\ast}$ with $\Omega_1^{\ast}$, and then weakly coupling the remaining (or active) region of $\Omega_0:=\Omega_0^{\ast}\backslash \Omega_1^{\ast}$ with $\Omega_1^{\ast}$ through their interface. While the integration on cut elements has been discussed in \cite{ref:antolin19a} (see also \cite{ref:rank12, ref:ruess14, ref:kudela16}), here we focus on dealing with interfaces, which includes creating a quadrature mesh for each interface as well as stabilizing bad cut elements that are adjacent to the union interface. In what follows, we explain the related algorithms in 2D and will also comment on the extension to 3D.

\subsection{Generation of the Interface Quadrature Mesh}
\label{sec:quad}

The key to creating an interface quadrature mesh is to find mesh intersections on the interface; see a 2D example in Figure \ref{fig:intersect2d}. Each visible interface is shared by two patches, one on the top and the other on the bottom. Recall that we denote the interface, the top patch and the bottom patch as $\Gamma_{ij}$, $\Omega_i$ and $\Omega_j$ ($i>j$), respectively. According to our construction of unions,  $\Gamma_{ij}$ is always part of the boundary of $\Omega_i$, so its geometric mapping is the same as that of $\Omega_i$ and it naturally has the mesh information of $\Omega_i$. Now the aim is to find out how the mesh of $\Omega_j$ intersects with $\Gamma_{ij}$.


\begin{figure}[htp]
\centering
\includegraphics[width=0.8\linewidth]{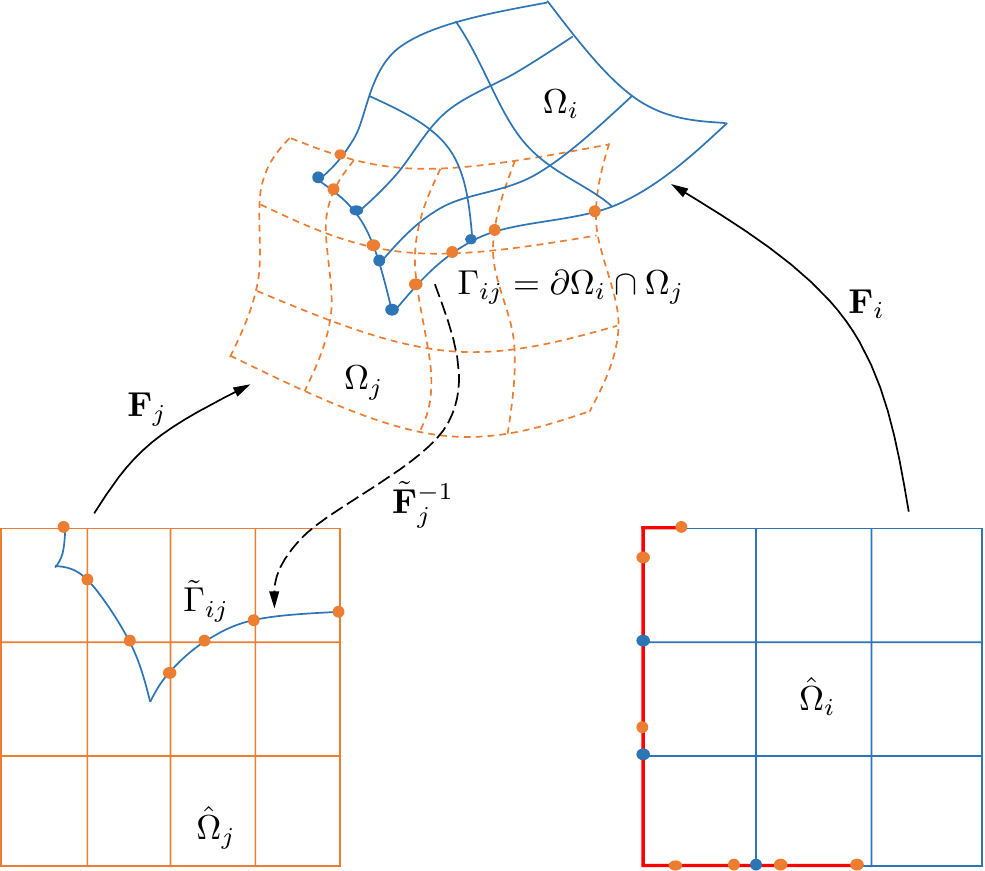}
\caption{Mesh intersections of a top patch (blue lines) and a bottom patch (orange lines) on their interface. The mesh intersections are marked in blue and orange dots. The parametric domain of the interface is marked as red lines.}
\label{fig:intersect2d}
\end{figure}

We find these intersections primarily in the parametric domain $\hat{\Omega}_j$ of $\Omega_j$ following three steps. First, we obtain an approximate preimage of $\Gamma_{ij}$ with respect to $\hat{\Omega}_j$ through the inversion algorithm \cite{ref:nurbsbook}, i.e., $\tilde{\Gamma}_{ij} := \tilde{\mathbf{F}}_j^{-1}(\Gamma_{ij})$, where the tilde indicates the approximation nature of the inversion algorithm.  Second, in $\hat{\Omega}_j$ we find the intersections of $\tilde{\Gamma}_{ij}$ with the axis-aligned knot lines of $\hat{\Omega}_j$, which is easier compared to the general curve-curve intersection; see the bottom left figure in Figure \ref{fig:intersect2d}. Third, we map these intersections to the physical domain through $\mathbf{F}_j$. However, the resulting points generally do not lie on $\Gamma_{ij}$, again due to the approximation of the inversion algorithm. Therefore, we further project these points onto $\Gamma_{ij}$ to get the final intersections; see orange and blue dots in Figure \ref{fig:intersect2d}. 

Now with all the intersections, we are ready to create the interface quadrature mesh and to compute the interface integral. We here emphasize two aspects that can improve the computation in terms of accuracy and efficiency. First, as the Nitsche's method needs the normal information of $\Gamma_{ij}$, we compute it using $\Omega_i$ where the geometric information is exact. In contrast, if we compute it using $\Omega_j$, we will lose accuracy because it relies on $\tilde{\Gamma}_{ij}$, which is only an approximation. Second, when evaluating basis functions from $\Omega_j$ in the interface integral, it involves to find corresponding quadrature points in $\hat{\Omega}_j$ through the approximate inverse mapping $\tilde{\mathbf{F}}_j^{-1}$, so there is an unnecessary coupling of the geometric operation (i.e., the inversion operation) and analysis. Instead, we precompute and store those quadrature points in $\hat{\Omega}_j$ such that later analysis can be performed without repeatedly appealing to the inversion operation.

\begin{remark} 
We use OpenCASCADE \cite{ref:opencascade}, an open source CAD system, to perform surface geometric operations, including creating the union of multiple spline patches and finding the mesh intersections on each interface. To the authors' knowledge, the default geometric tolerance in OpenCASCADE is around $10^{-8}$ in trimming-related operations and cannot be further reduced. Its influence will be seen in a numerical example in Section \ref{sec:result}, where the geometric error induced by this tolerance begins to dominate once the approximation error (in $L^2$ norm) reaches $10^{-8}$. 
\end{remark}


\subsection{Implementation of Minimal Stabilization}
\label{sec:stab}

The minimal stabilization method introduced in Section \ref{sec:ministab} mainly needs to: (1) find a list of bad-to-good element pairs, and (2) replace the basis functions of each bad element with extended polynomials from its good neighbor. In the following, we explain the procedure in 2D terminologies but extension to 3D is straightforward.

Simply speaking, a bad element is a cut element with ``small" effective area on which we have to compute fluxes, i.e., normal derivatives at its boundary. All the other active elements are good elements. In practice, we compute element areas in the parametric domain and use a given threshold to identify bad elements, which serves as an approximate criterion to Definition \eqref{def_goodandbad}. Then we follow the procedure described in Section \ref{sec:ministab}.


Next, we take a look at the interface integral that contributes to the stiffness matrix. Let $\tau(\Gamma_{ij})$ denote the quadrature mesh of the interface $\Gamma_{ij}$, and $e\in\tau(\Gamma_{ij})$ be an quadrature element. The two adjacent elements to $e$ are denoted as $K_i^e\in\Omega_i$ and $K_j^e\in\Omega_j$. The index set of basis functions $B_{\bm{k},i}$ with support on $K_i^e$ is denoted as $I_{K_i^e}$; similarly, $I_{K_j^e}$ corresponds to $K_j^e$. Note that we neglect the degree information in the notation of basis functions as it is fixed once the patch index $i$ or $j$ is given. We are particularly interested in the terms involving normal derivatives, and such a term takes the following form when the one-sided flux from the top patch $\Omega_i$ is used,
\begin{equation}
\int_{e} \pd{B_{\bm{k},i}}{n_i} \, ( B_{\bm{l},i} - B_{\bm{m},j} ),
\label{eq:intf_impl_0}
\end{equation}
where $\bm{k}, \bm{l}\in I_{K_i^e}$, and $\bm{m}\in I_{K_j^e}$. The stability issue originates from $\partial B_{\bm{k},i} / \partial n_i$ if $K_i^e$ is badly cut; otherwise Equation \eqref{eq:intf_impl_0} contributes to the matrix entries corresponding to the indices $(\bm{k},\bm{l})$ and $(\bm{k},\bm{m})$. 

In the following, we focus on the case that $K_i^e$ is badly cut. The minimal stabilization consists in replacing $\partial B_{\bm{k},i} / \partial n_i$ with a stabilized version that involves function extension from the good neighbor $(K_i^e)'$ of $K_i^e$. In other words, we need to extend basis functions defined on $(K_i^e)'$ to $K_i^e$ and use the extended functions to evaluate the involved normal derivatives. Specifically, we follow three steps. First, we find the Cartesian bounding box $(K_i^e)_b'$ of $(K_i^e)'$ in the physical domain and define on it a set of bi-degree-$p$ Bernstein polynomials $\mathsf{b}_{l}(\bm x)$, where $l\in\{1,2,\ldots,(p+1)^2\}$, $\bm{x}\in (K_i^e)_b'$, and $p$ is the highest degree in $(K_i^e)_b'$. Second, let $I_{(K_i^e)'}$ be the index set of basis functions with support on $(K_i^e)'$, and we compute a $L^2$ projection of each $B_{\bm{k}',i}$ ($\bm{k}'\in I_{(K_i^e)'}$) using these Bernstein polynomials. As a result, we have
\begin{equation*}
P_{(K_i^e)'}(B_{\bm{k}',i}(\bm x)) = \sum_{l=1}^{(p+1)^2} c_{\bm{k}' l}\, \mathsf{b}_{l}(\bm x), \quad \bm{x}\in (K_i^e)_b',
\end{equation*}
where $P_{(K_i^e)'}$ stands for the $L^2$-orthogonal projection onto $\mathbb{Q}_p\left((K_i^e)'\right)$. Note that $c_{\bm{k}' l}\in\mathbb{R}$ are obtained by solving a local system of linear equations $\sum_{l=1}^{(p+1)^2} M_{ml} c_{\bm{k}' l} = F_m$ for $m=1,\ldots,(p+1)^2$, where
\begin{equation*}
M_{ml}  = \int_{(K_i^e)'} \mathsf{b}_{m} \, \mathsf{b}_{l}  \quad\text{and}\quad F_m = \int_{(K_i^e)'} B_{\bm{k}',i} \, \mathsf{b}_{m} . 
\end{equation*}
Third, we define the extension of $B_{\bm{k}',i}$ to be $P_{(K_i^e)'}(B_{\bm{k}',i})$ and enlarge the definition domain of the Bernstein polynomials by including the bounding box $(K_i^e)_b$ of the bad element $K_i^e$ as well, that is,
\begin{equation*}
\mathcal{E}(P_{(K_i^e)'}(B_{\bm{k}',i}(\bm x))) := \sum_{l=1}^{(p+1)^2} c_{\bm{k}' l} \mathsf{b}_{l}(\bm x), \quad \bm{x}\in (K_i^e)_b' \cup (K_i^e)_b.
\end{equation*}
Finally, the stabilized interface integral corresponding to Equation \eqref{eq:intf_impl_0} becomes
\begin{equation*}
\int_{e} \pd{\mathcal{E}(P_{(K_i^e)'}(B_{\bm{k}',i}))}{n_i} \, ( B_{\bm{l},i} - B_{\bm{m},j} ) ,
\label{eq:intf_impl_stab}
\end{equation*}
where recall that $\bm{k}'\in I_{(K_i^e)'}$, $\bm{l}\in I_{K_i^e}$, and $\bm{m}\in I_{K_j^e}$. Therefore, when $K_i^e$ is a bad element, it is Equation \eqref{eq:intf_impl_stab} rather than Equation \eqref{eq:intf_impl_0} that contributes to the stiffness matrix. Particularly, it contributes to the entries corresponding to the indices $(\bm{k}',\bm{l})$ and $(\bm{k}',\bm{m})$.

We have discussed the stabilization with the one-sided flux. Following a similar procedure, we can obtain the stabilization with the symmetric average flux as well, which, however, generally requires to stabilize more elements. More specifically, with the symmetric average flux, the term we need to stabilize becomes
\begin{equation*}
\int_{e} \frac{1}{2} \left( \pd{B_{\bm{k},i}}{n_i} +  \pd{B_{\bm{n},j}}{n_i} \right) \, ( B_{\bm{l},i}-B_{\bm{m},j} ) ,
\end{equation*}
where $\bm{n}\in I_{K_j^e}$ and the integral involves normal derivatives from both patches. Therefore, in addition to $\partial B_{\bm{k},i} / \partial n_i$, the flux $\partial B_{\bm{n},j} / \partial n_i$ also needs stabilization if $K_j^e$ is badly cut. 

\begin{remark}
We have implemented the stabilized union operation on top of igatools \cite{ref:igatools}, an open source isogeometric library written in C++, also with the help of OpenCASCADE \cite{ref:opencascade} and IRIT \cite{ref:irit} for the geometric operations.
\end{remark}

\begin{remark}
Conditioning is another important issue related to trimming. Guarantee of stability does not necessarily imply a well-conditioned stiffness matrix due to the presence of cut basis functions \cite{ref:buffa19}. A proper preconditioner is needed to ensure a reliable solution. The simple diagonal scaling preconditioner was used in our previous work  \cite{ref:antolin19a} on isogeometric V-rep, where the numerical results showed that it works in solving the Poisson's equation and linear elasticity problems. In this work, we also use this simple technique to deal with the conditioning issue. Another alternative is a recent work on the multigrid preconditioner \cite{ref:prenter19}, which can deliver cut-element independent convergence rates in the context of immersed isogeometric analysis. However, further investigation, especially theoretically, is needed to advance our knowledge on this challenging issue.
\end{remark}

\section{Numerical Examples}
\label{sec:result}

\begin{figure}[htb]
\centering
\begin{tabular}{cc}
\includegraphics[width=5cm]{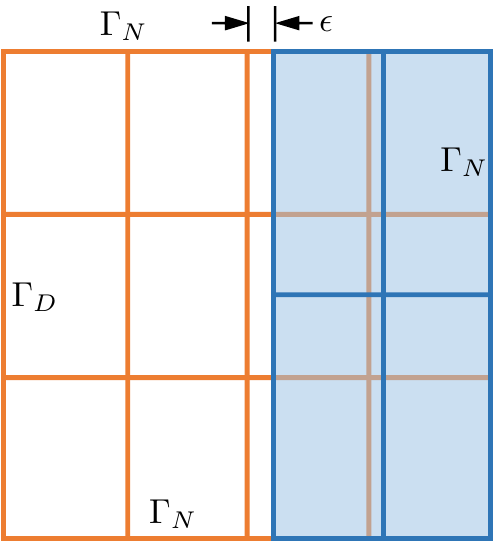} &
\includegraphics[width=5.5cm]{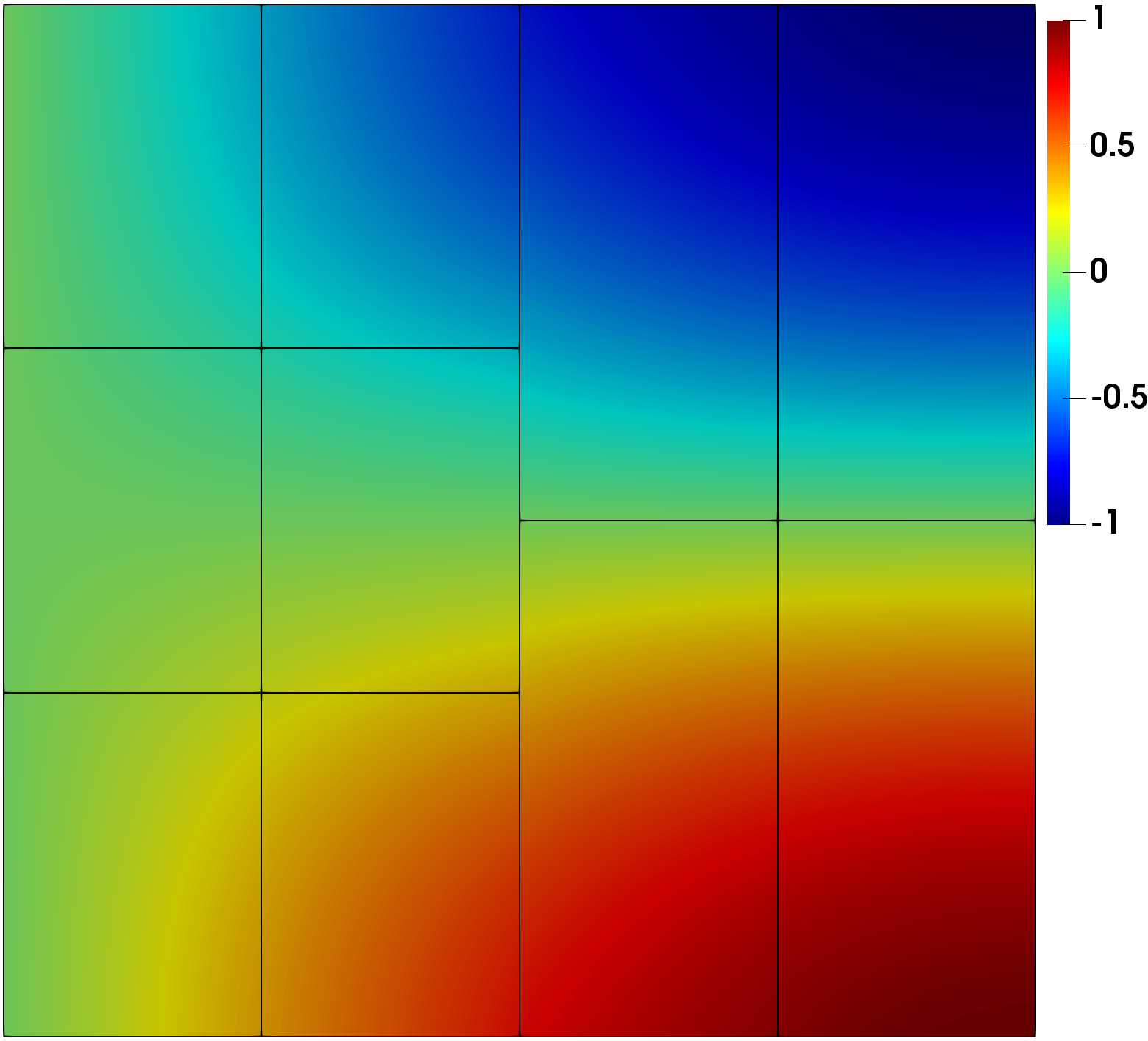} \\
(a) & (b)\\
\includegraphics[width=0.45\linewidth]{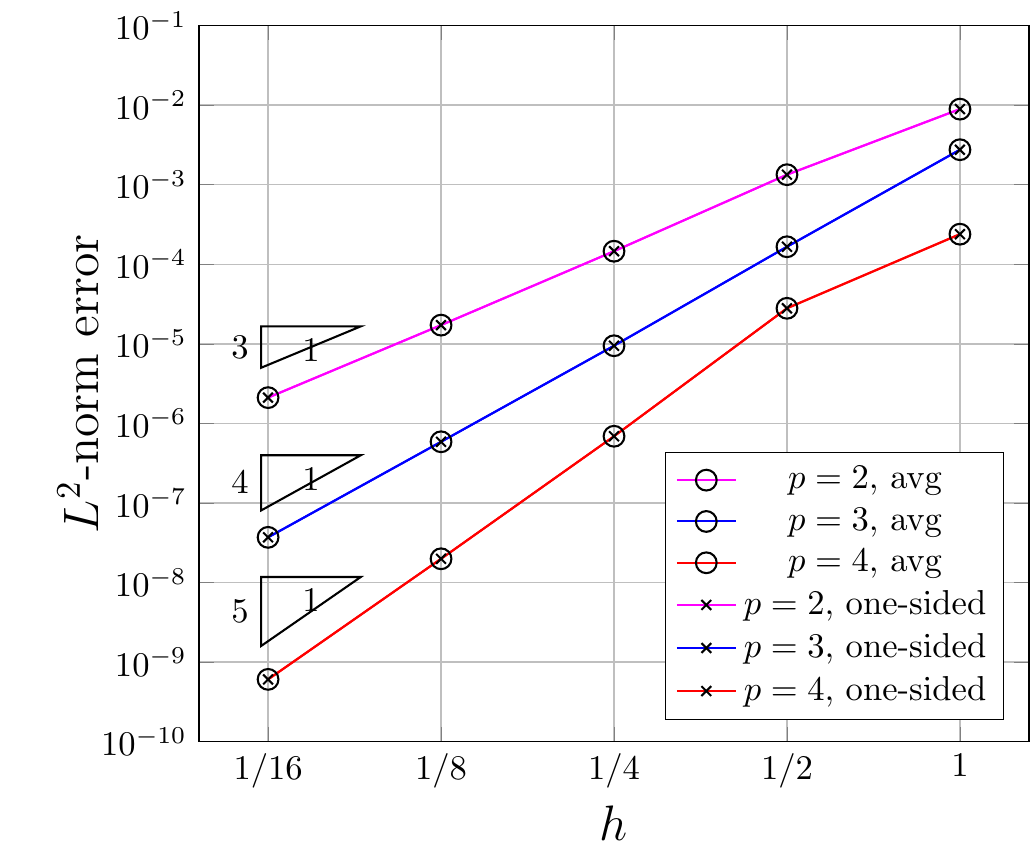} &
\includegraphics[width=0.45\linewidth]{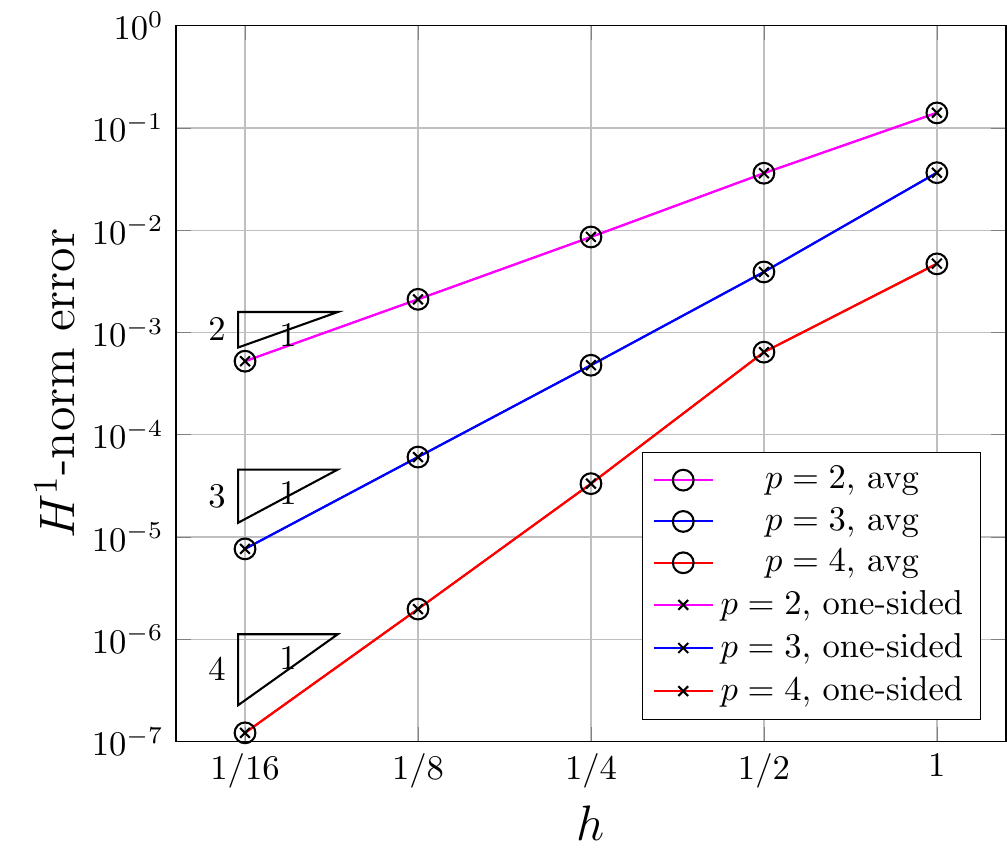} \\
(c) & (d)
\end{tabular}
\caption{The unit square example by a two-patch union. (a) Meshes of the two patches (orange and blue lines) as well as Dirichlet ($\Gamma_D$) and Neumann ($\Gamma_N$) boundary conditions, (b) the solution field on the input mesh in (a) using the quadratic basis, (c) the convergence plot in the $L^2$-norm error, and (d) the convergence plot in the $H^1$-norm error. $\epsilon$ in (b, c, d) is set to be $10^{-6}$.}
\label{fig:ex_square}
\end{figure}

In this section, we present three examples to demonstrate the convergence and conditioning by solving the Poisson's equation on various domains obtained through the union operation. We then show the geometric flexibility of our proposed method by solving the linear elasticity problem on a more complex 2D geometry. In all the numerical tests, we set the penalty parameter $\beta$ in Nitsche's formulation as $6p^2$, where $p$ is the degree of the spline discretization. The area-ratio threshold is set to be $10\%$ to identify bad elements.

\subsection{Convergence and Conditioning Under Bad Trimming}
We start with a two-patch union that forms a unit square. This simple test is meant to show that the minimal stabilization works robustly even when there are extremely small cut elements involved. As shown in Figure \ref{fig:ex_square}(a), the bottom patch ($\Omega_0^{\ast}$) is a unit square $[0,1]^2$ with a $4\times 3$ mesh (orange lines), whereas the top patch ($\Omega_1^{\ast}$) covers the region $[0.5+\epsilon,1]\times [0,1]$ with a $2\times 2$ mesh (blue lines). The parameter $\epsilon\in(0,10^{-2})$ controls trimming of the bottom patch, or equivalently, $\epsilon$ is the width of the cut elements in the bottom patch. Both patches are B-spline patches. Particularly, we set $\epsilon=10^{-6}$ to perform a convergence study with bases of degrees 2, 3 and 4 in all the patches. We consider the manufactured solution: $u(x,y) = \sin(\pi x/2) \cos(\pi y),\ (x,y)\in [0,1]^2$, with homogeneous Dirichlet and Neumann boundary conditions imposed according to Figure \ref{fig:ex_square}(a). Figure \ref{fig:ex_square}(b) shows the solution obtained on the input mesh using the quadratic basis, where the cut elements are invisible due to their small scale. The convergence plots in $L^2$- and $H^1$-norm error are shown in Figure \ref{fig:ex_square}(c, d), where we observe expected optimal convergence rates in both norms. 

Note that we have used two types of fluxes $\left<\pd{v_h}{n}\right>_t$ ($t=\frac{1}{2},1$) to show that they behave almost the same in terms of convergence and conditioning. We should also note that the symmetric average flux involves the flux from the bad cut elements and thus needs to be stabilized through the minimal stabilization method. On the other hand, the one-sided flux comes from a non-cut domain so it does not need stabilization. Indeed, we observe in Figure \ref{fig:ex_square} that the convergence curves are indistinguishable using both types of fluxes.

Next, we study the conditioning of global stiffness matrices in three cases: (1) the symmetric average flux without stabilization, (2) the symmetric average flux with stabilization, and (3) the one-sided flux (no need for stabilization). We compute the condition number of the rescaled stiffness matrix $D_s^{-1/2} K_s D_s^{-1/2}$, where $K_s$ is the stiffness matrix and $D_s$ denotes $\mathrm{diag}(K_s)$.

\begin{figure}[htb]
\centering
\begin{tabular}{ccc}
\includegraphics[width=0.3\linewidth]{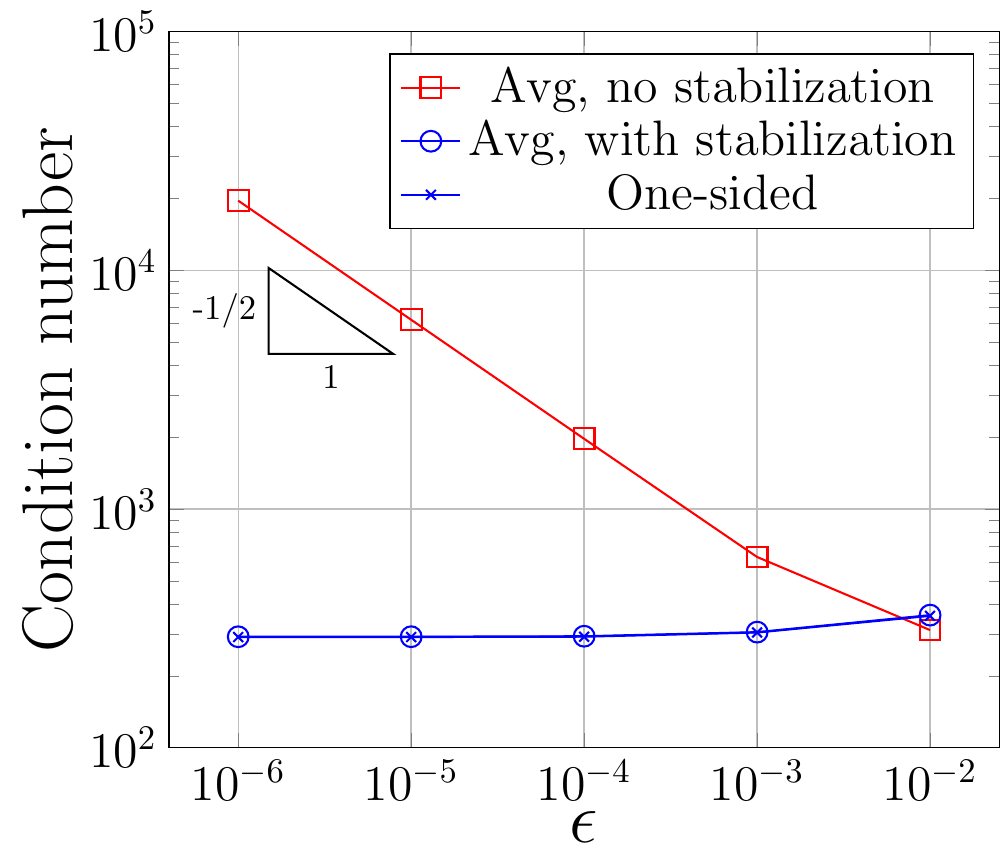} &
\includegraphics[width=0.3\linewidth]{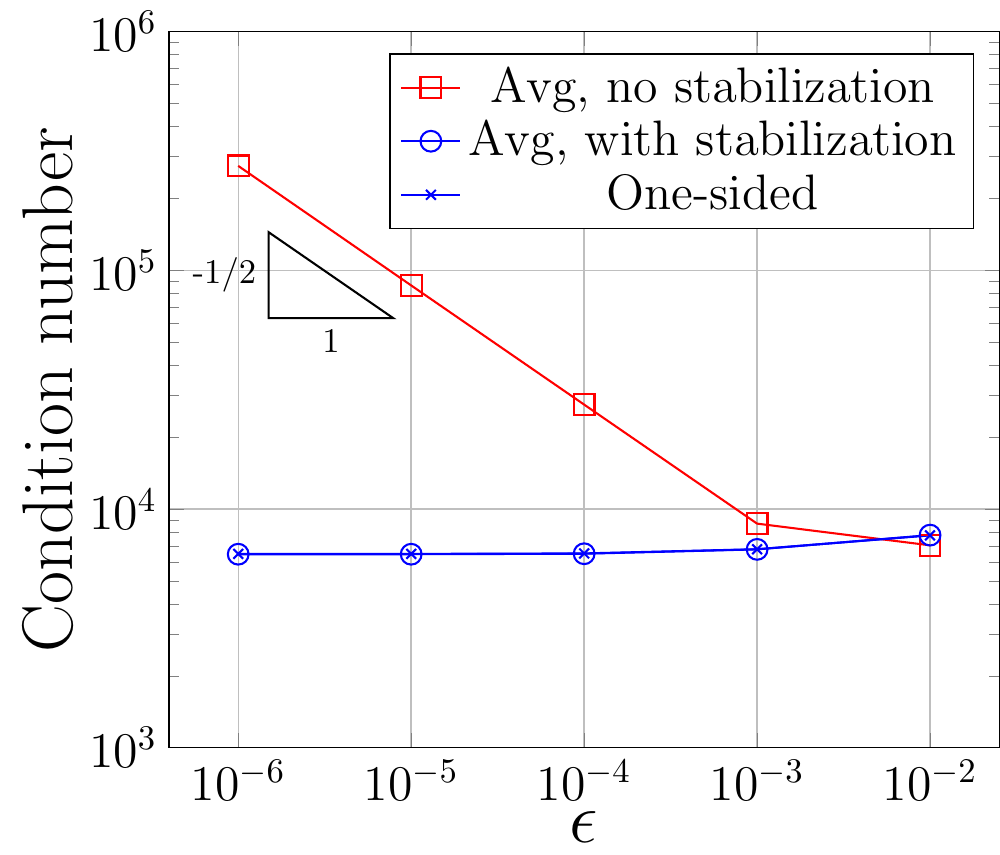} &
\includegraphics[width=0.3\linewidth]{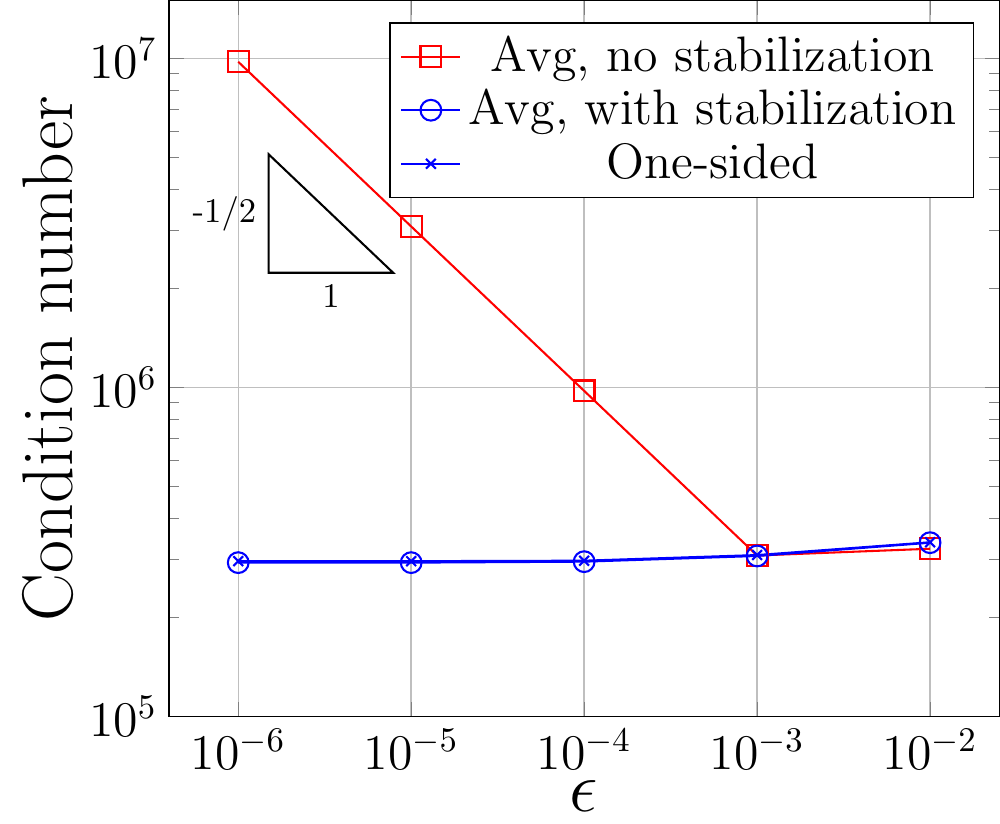} \\
(a) $p=2$ & (b) $p=3$ & (c) $p=4$ \\
\end{tabular}
\caption{Condition numbers with respect to the width of cut elements, $\epsilon$, see Figure \ref{fig:ex_square}(a). The condition numbers are obtained using bases of degree 2, 3 and 4 on the input mesh shown in Figure \ref{fig:ex_square}(a).}
\label{fig:square_cond_eps}
\end{figure}

We test both the influence of $\epsilon$ on a fixed mesh and the influence of the mesh size $h$ with a fixed small $\epsilon$. First, given the input mesh shown in Figure \ref{fig:ex_square}(a) and bases of different degrees (2, 3 and 4), we compute their corresponding condition numbers changing $\epsilon$ from $10^{-2}$ down to $10^{-6}$. The result is summarized in Figure \ref{fig:square_cond_eps}. We observe that in Cases (2) and (3), the condition number is independent from trimming, that is, it almost remains constant as $\epsilon$ decreases. Moreover, the condition numbers obtained in these two cases are indistinguishable. In contrast, the condition number in Case (1) increases exponentially in the power of $1/2$ (see \cite{DEPRENTER2017297} for a more involved discussion about the dependence of the condition number on $\eps$).

\begin{figure}[htb]
\centering
\begin{tabular}{ccc}
\includegraphics[width=0.3\linewidth]{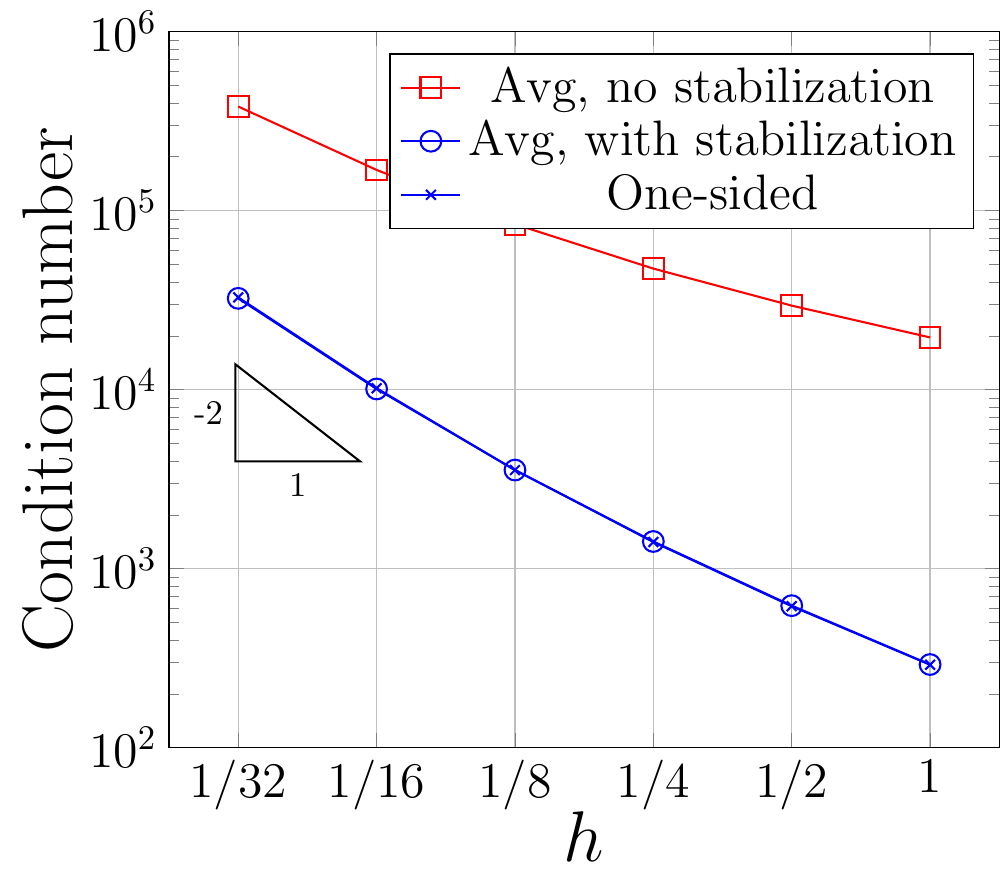} &
\includegraphics[width=0.3\linewidth]{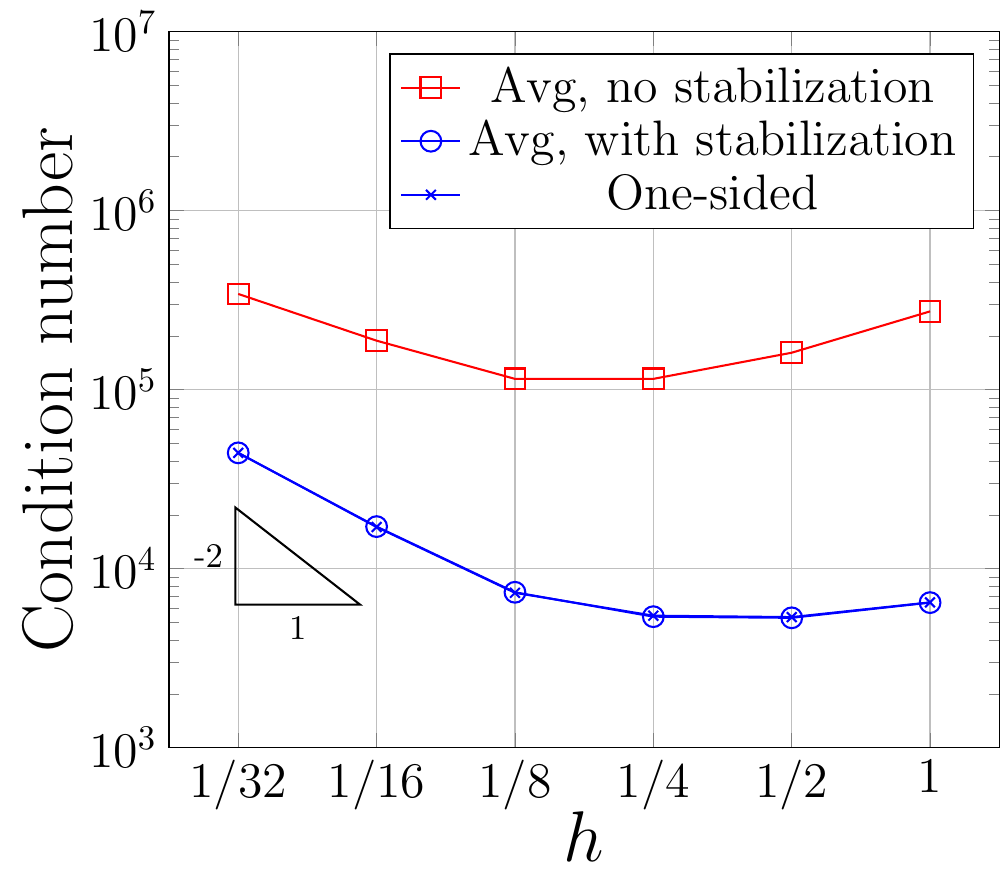} &
\includegraphics[width=0.3\linewidth]{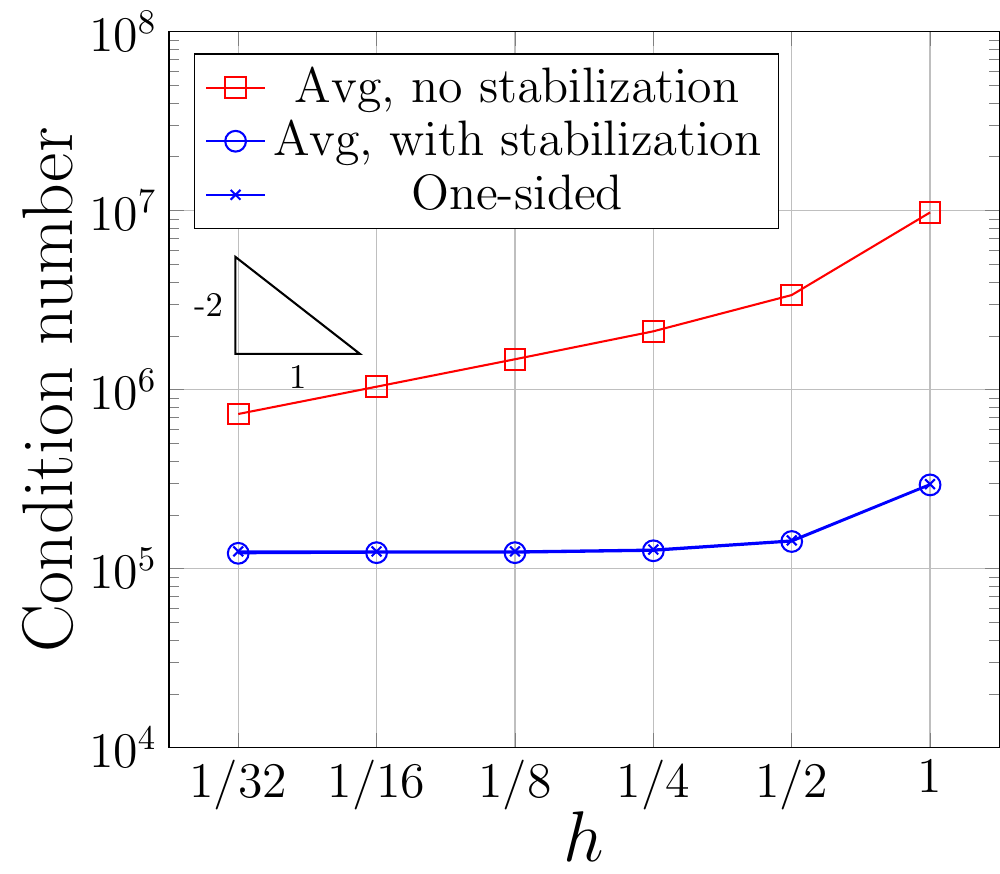} \\
(a) $p=2$ & (b) $p=3$ & (c) $p=4$ \\
\end{tabular}
\caption{Condition numbers with respect to the mesh size indicator $h$. The condition numbers are obtained using bases of degree 2, 3 and 4 on a series of refined meshes, given $\epsilon=10^{-6}$.}
\label{fig:square_cond_h}
\end{figure}

We now fix $\epsilon$ to be $10^{-6}$ in the initial mesh and change the mesh size $h$ (via global refinement) to further compare conditioning. First, we observe in Figure \ref{fig:square_cond_h} that for all the degrees considered, the condition number in Cases (2) and (3) is constantly lower than that in Case (1). Second, higher-degree splines generally yield higher condition numbers under the same mesh size in all the cases. Third, in the low degree case (e.g., $p=2$), the condition number tends to be controlled by the mesh size $h$, and it increases in the order of $h^{-2}$ as $h$ decreases, as it is expected. On the other hand, in the high degree case (e.g., $p=4$), the condition number is more controlled by the size $\epsilon$ of cut elements. As $h$ goes down, the effective area ratio of a cut element actually becomes larger, and this is why the condition number in Case (1) decreases as $h$ decreases, whereas it remains almost constant in Cases (2) and (3); see Figure \ref{fig:square_cond_h}(c). 

In both the convergence test and condition test, we have shown that the one-sided flux works almost the same as the symmetric average flux with stabilization. The one-sided flux is chosen in the following tests as it generally needs to stabilize fewer elements than the symmetric average flux case.

\begin{figure}[htb]
\centering
\begin{tabular}{ccc}
\includegraphics[width=0.3\linewidth]{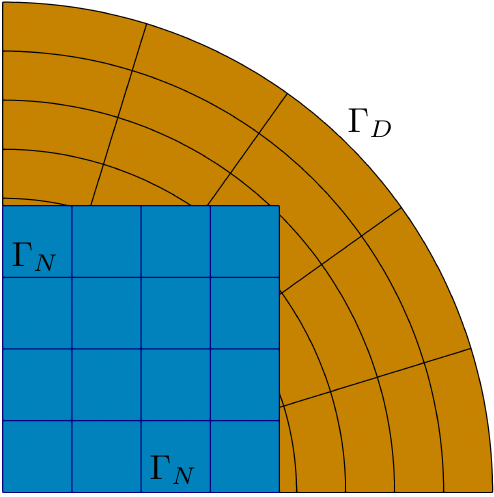} &
\includegraphics[width=0.3\linewidth]{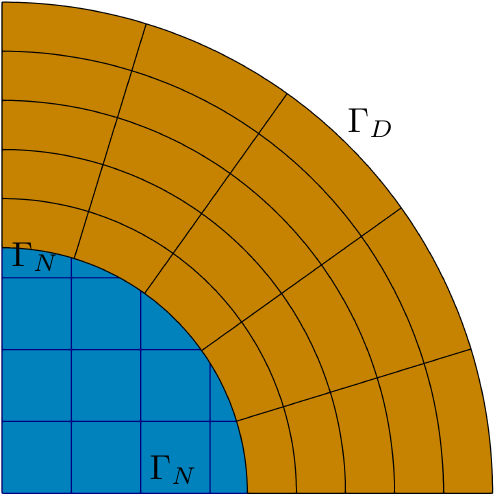} &
\includegraphics[width=0.3\linewidth]{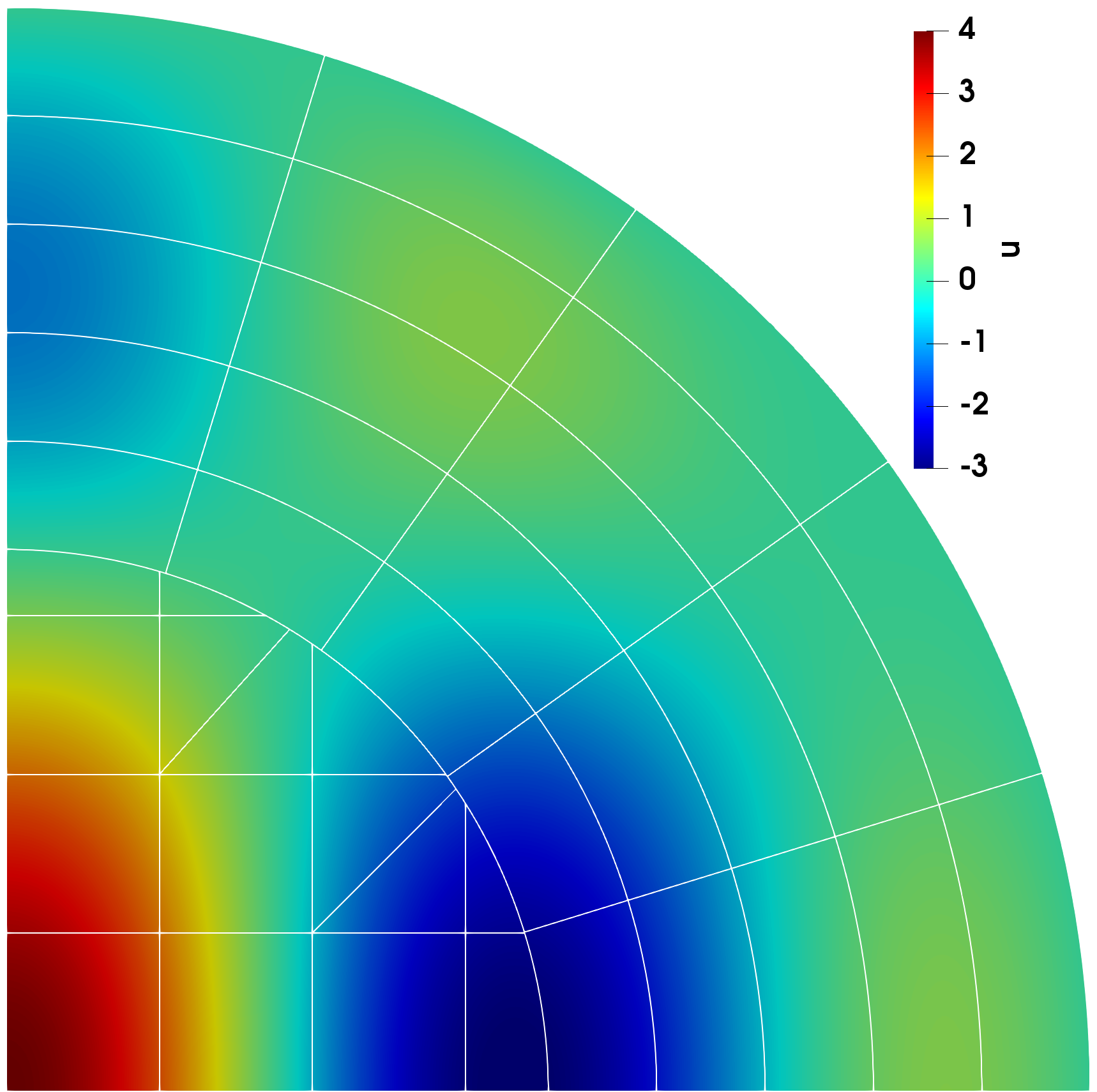} \\
(a) & (b) & (c)\\
\end{tabular}
\caption{The disk example by the union of an annulus with a rectangle. (a) The rectangle on top of the annulus, (b) the annulus on top of the rectangle, and (c) the solution field using the mesh in (b) with quartic splines, where the white lines represent both the B\'{e}zier mesh and the quadrature mesh for cut elements.}
\label{fig:ex_disk_geom}
\end{figure}

\begin{figure}[!htb]
\centering
\begin{tabular}{cc}
\includegraphics[width=0.45\linewidth]{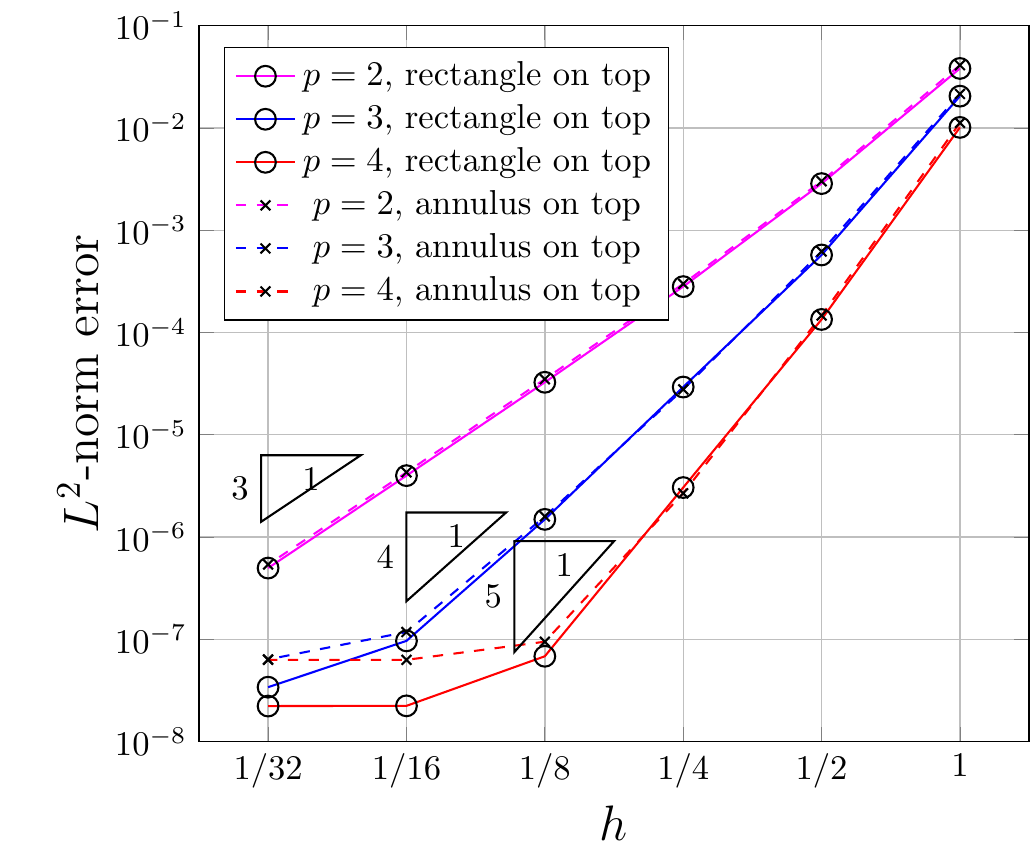} &
\includegraphics[width=0.45\linewidth]{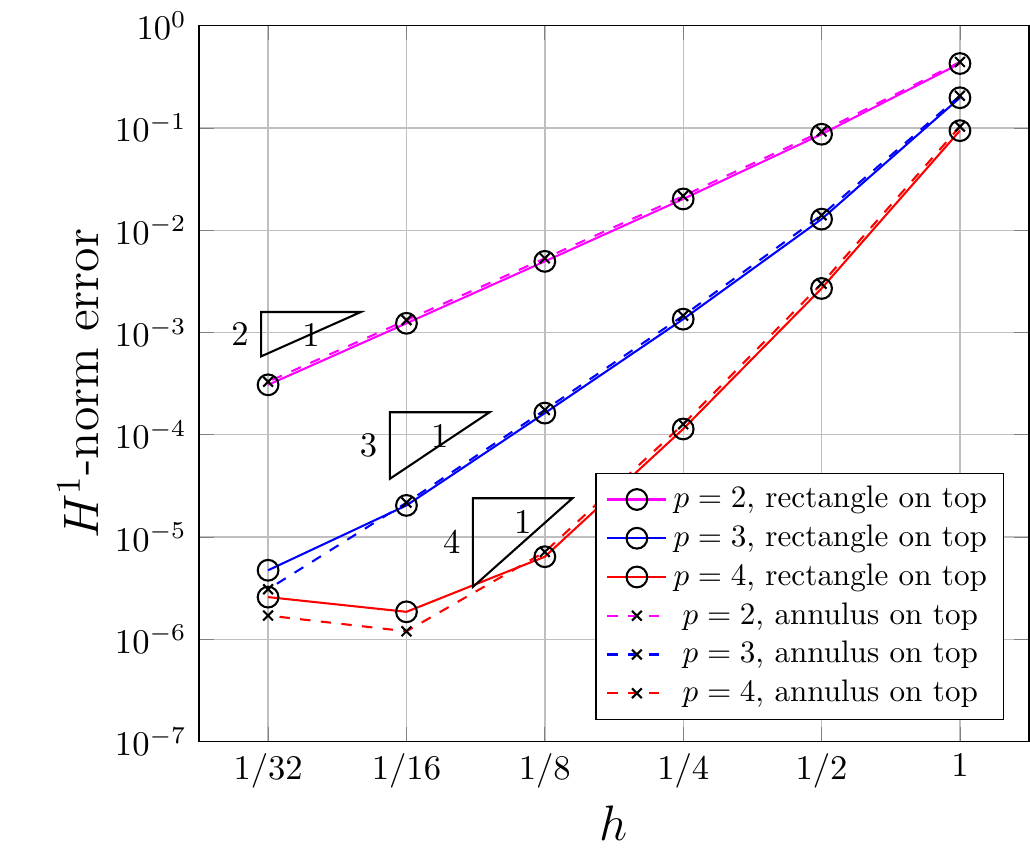} \\
(a) & (b)\\
\end{tabular}
\caption{Convergence plots of the disk example in the $L^2$-norm error (a) and the $H^1$-norm error (b), where the solid and dashed lines represent results corresponding to Figure \ref{fig:ex_disk_geom}(a, b), respectively.}
\label{fig:ex_disk_conv}
\end{figure}

\subsection{Influence of Patch Ordering}
We next study a disk geometry centered at $(0,0)$ with a radius of 2. It is formed by the union of an annulus with a rectangle, and we focus on a quarter of it due to symmetry. The annulus is represented by a NURBS patch with an inner radius of 1 and an outer radius of 2, which has a $5\times 5$ mesh. The rectangle is a B-spline patch covering the region $[0,1.13]\times [0,1.17]$ with a $4\times 4$ mesh. We consider two arrangements of patches to check if there is a difference in the numerical performance: (1) the rectangle on top of the annulus, and (2) the annulus on top of the rectangle; see Figure \ref{fig:ex_disk_geom}(a, b). In the convergence study, we take the manufactured solution $u(x,y) = (4-x^2-y^2)\cos(\pi x) \cos\left(\frac{\pi y}{2}\right)$ ($x\geq 0,\ y\geq 0,\ x^2+y^2\leq 4$), and use bases of degrees 2, 3 and 4 everywhere. Homogeneous Dirichlet and Neumann boundary conditions are shown in Figure \ref{fig:ex_disk_geom}(a, b). In Figure \ref{fig:ex_disk_conv}, we observe the expected optimal convergence in all the cases before the $L^2$-norm error reaches $10^{-8}$. Afterwards we observe a deteriorated behavior due to the dominance of the geometric error, which is induced by the fixed tolerance setting ($\sim 10^{-8}$) in OpenCASCADE. Moreover, in all the convergence plots, we do not find distinguishable differences in the two different arrangements before the $L^2$-norm error hits the geometric tolerance; compare dashed and solid lines.

\begin{figure}[htb]
\centering
\begin{tabular}{cc}
\includegraphics[width=0.45\linewidth]{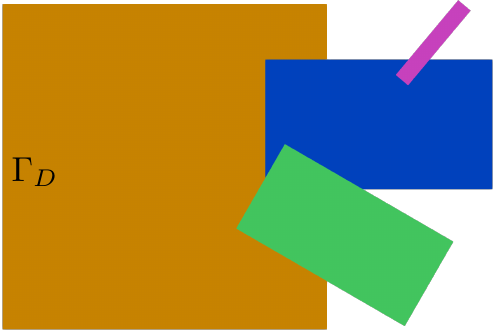} &
\includegraphics[width=0.45\linewidth]{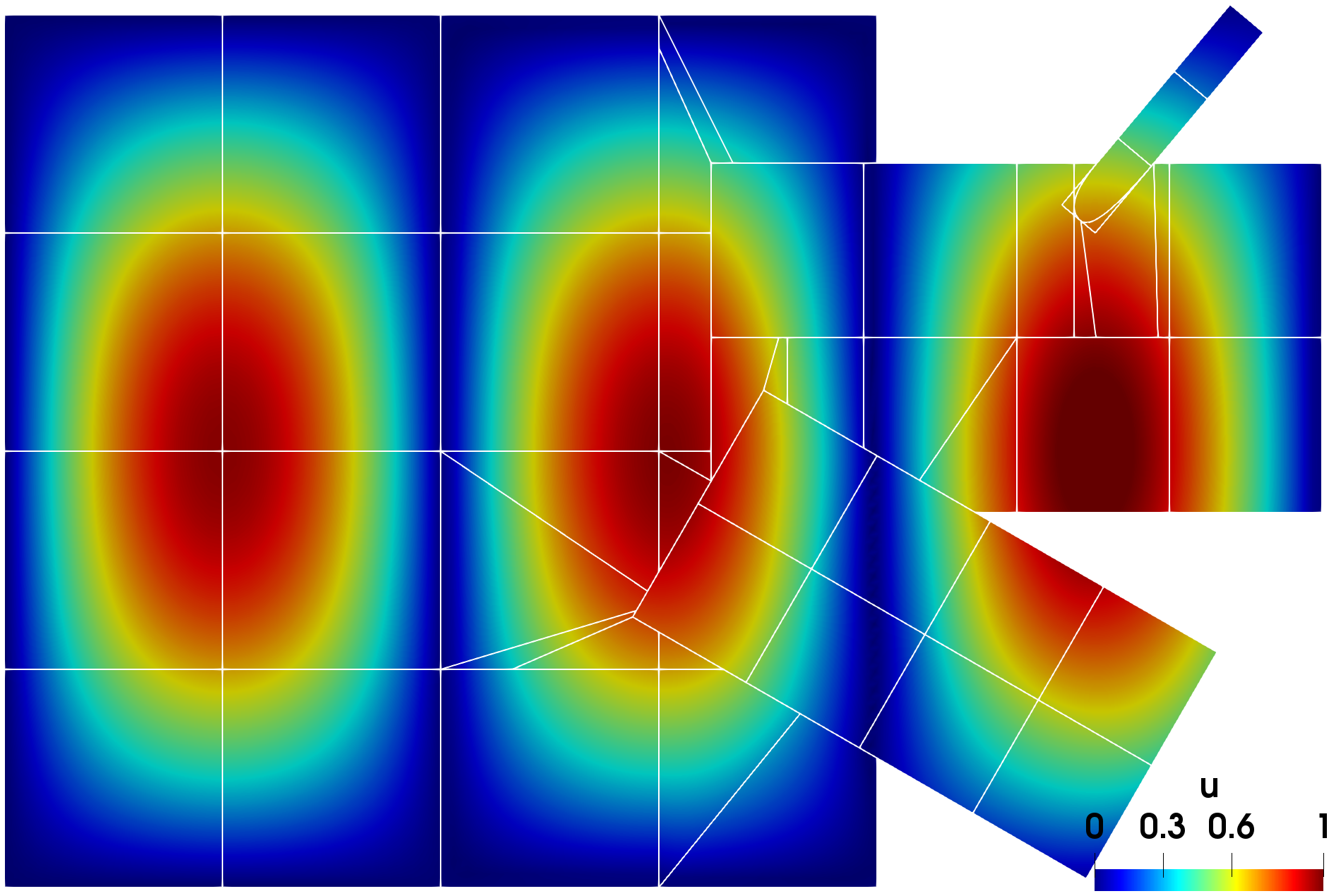} \\
(a) & (b)\\
\includegraphics[width=0.45\linewidth]{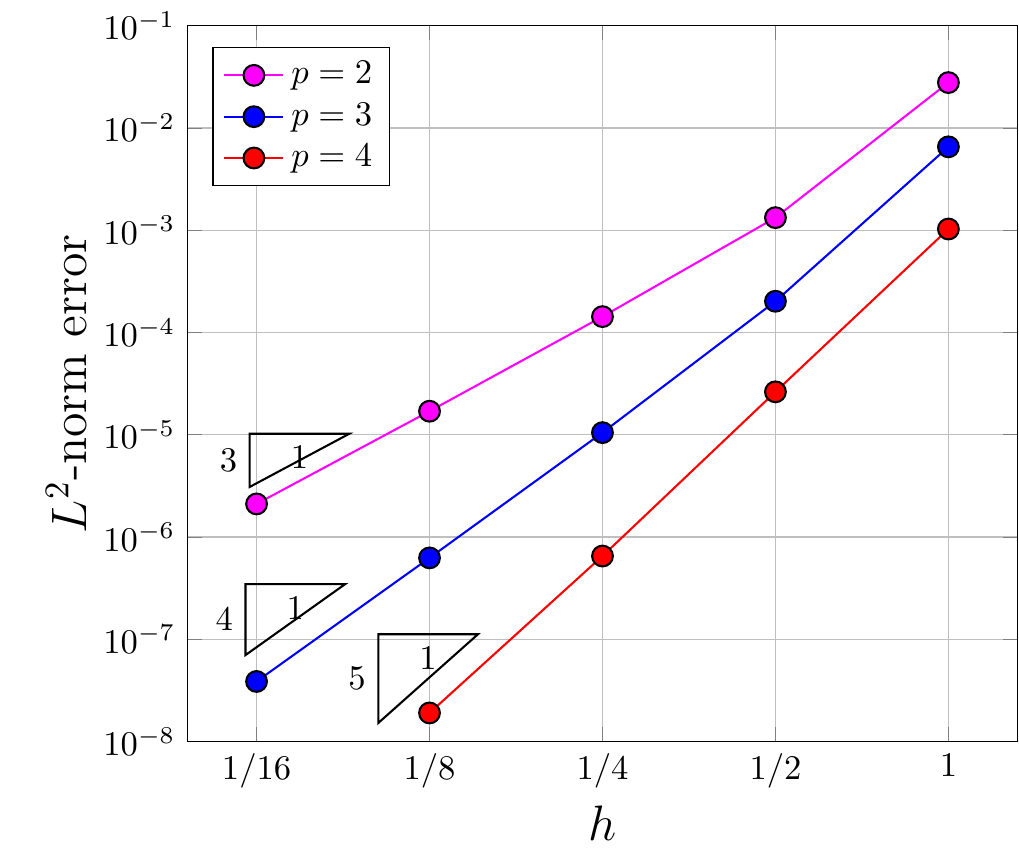} &
\includegraphics[width=0.45\linewidth]{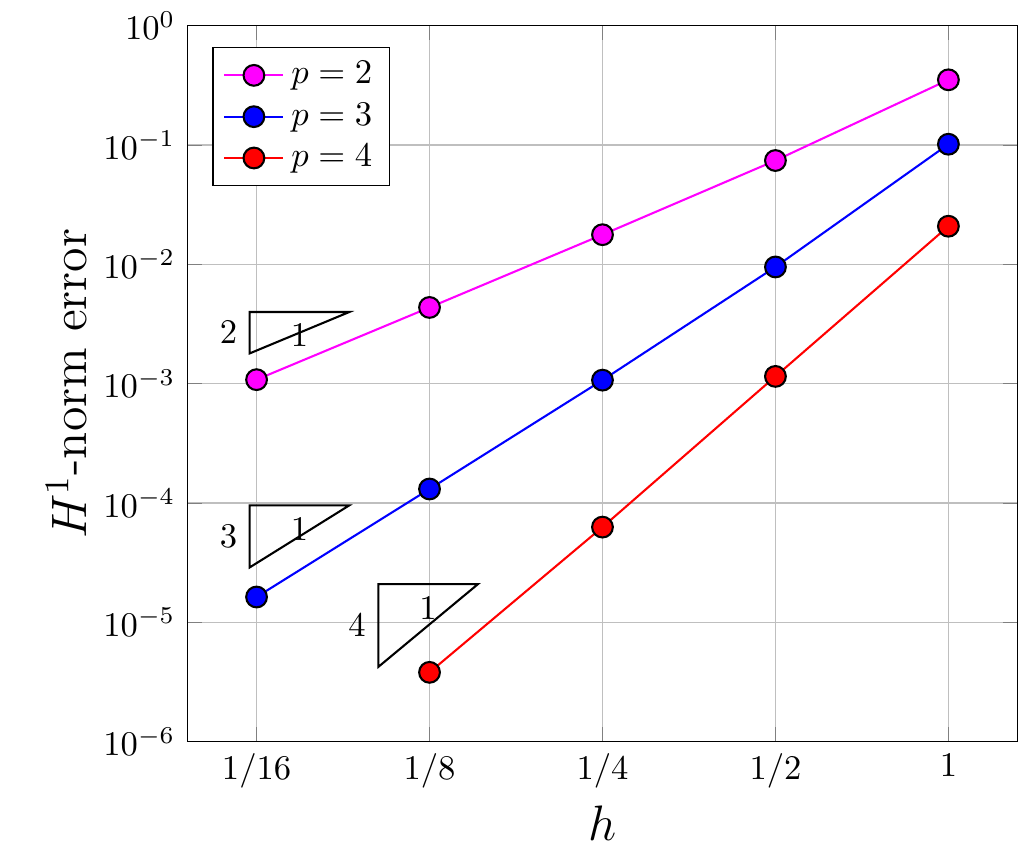} \\
(c) & (d)
\end{tabular}
\caption{The example of multiple overlapping patches. (a) Arrangement of the patches and the homogeneous boundary Dirichlet boundary condition, (b) the solution field on the input mesh using the quadratic basis, (c) the convergence plot in the $L^2$-norm error, and (d) the convergence plot in the $H^1$-norm error.}
\label{fig:three_patch}
\end{figure}

\subsection{Multiple Overlapping Patches}
We further study an example that involves multiple overlapping patches; see Figure \ref{fig:three_patch}(a). In particular, there is a region where three patches are overlapped; see the intersection region of orange, blue and green patches. All the patches are B-spline patches. We take the manufactured solution $u(x,y) = \sin(2\pi x) \sin(\pi y)$ for the convergence test, and we use bases of degrees 2, 3 and 4 to solve the Poisson's problem. The homogeneous Dirichlet boundary condition is imposed according to Figure \ref{fig:three_patch}(a), whereas the Neumann boundary condition is imposed on all the other boundaries. Again, we observe the expected optimal convergence for all the degree considered; see Figure \ref{fig:three_patch}(c, d). Note that the blue patch provides the one-sided flux to the orange patch, and it is in the meanwhile cut by the green patch, so generally it needs stabilization. Recall that we set the area-ratio threshold to be $10\%$. In our test cases, we observe that usually around $3\%$ to $7\%$ of cut elements are identified as bad elements. In other words, the stabilization is only needed for a small number of elements.

\subsection{A Complex Geometry Obtained via Boolean Operations}
As the last example, we consider a more complex geometry, a toy car wheel model in the planar domain as shown in Figure \ref{fig:wheel}, to show the potential capability of the proposed method. Such a geometry can be easily created with a combination of both trimming and union operations, more specifically, by first generating two annuli, putting handles on top of them via union, and finally creating holes of different sizes via trimming. Two boundary conditions are shown in Figure \ref{fig:wheel}(a), whereas the homogeneous Neumann boundary condition is imposed on all the other boundaries. We use quadratic splines to solve the linear elasticity problem on a series of meshes under the plane strain assumption, where the material is homogeneous and isotropic with the Young's modulus and the Poisson's ratio being $1$ and $0.3$, respectively. In particular, we show the displacement field on the initial mesh and the von Mises stress on the mesh after three times of global refinement; see Figure \ref{fig:wheel}(b, c), respectively. As expected, we observe stress concentrations around holes as well as sharp corners. Moreover, our computational tool is indeed robust in reparameterizing cut elements and handling union interfaces, and it can be easily adapted to solving different elliptic PDEs.

\begin{figure}[htb]
\centering
\begin{tabular}{ccc}
\includegraphics[width=0.3\linewidth]{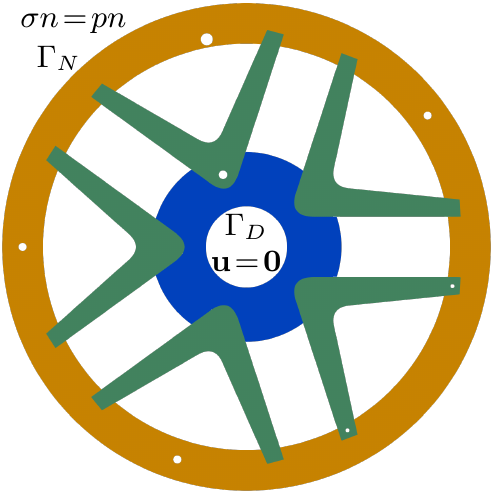} &
\includegraphics[width=0.3\linewidth]{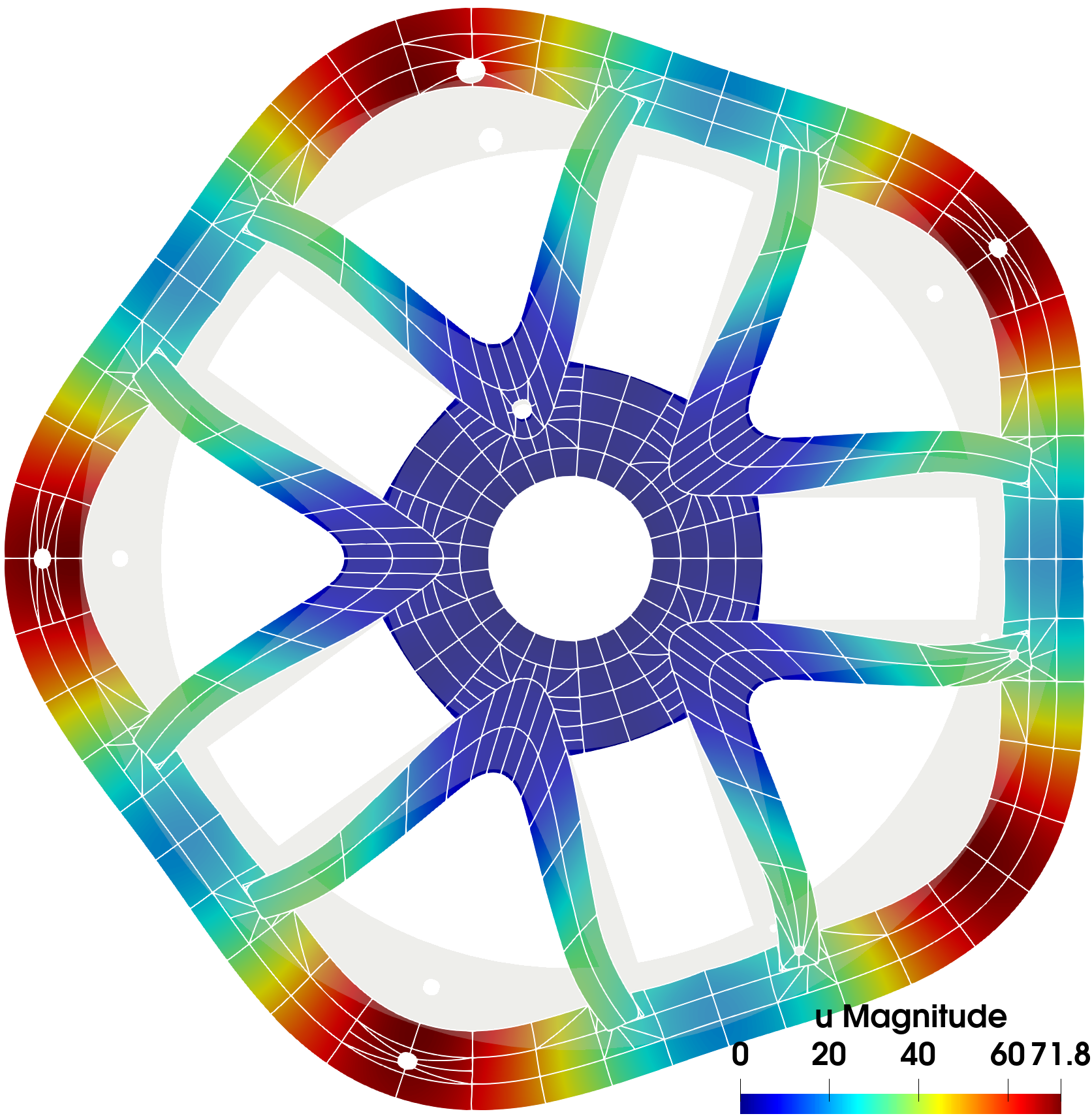} &
\includegraphics[width=0.3\linewidth]{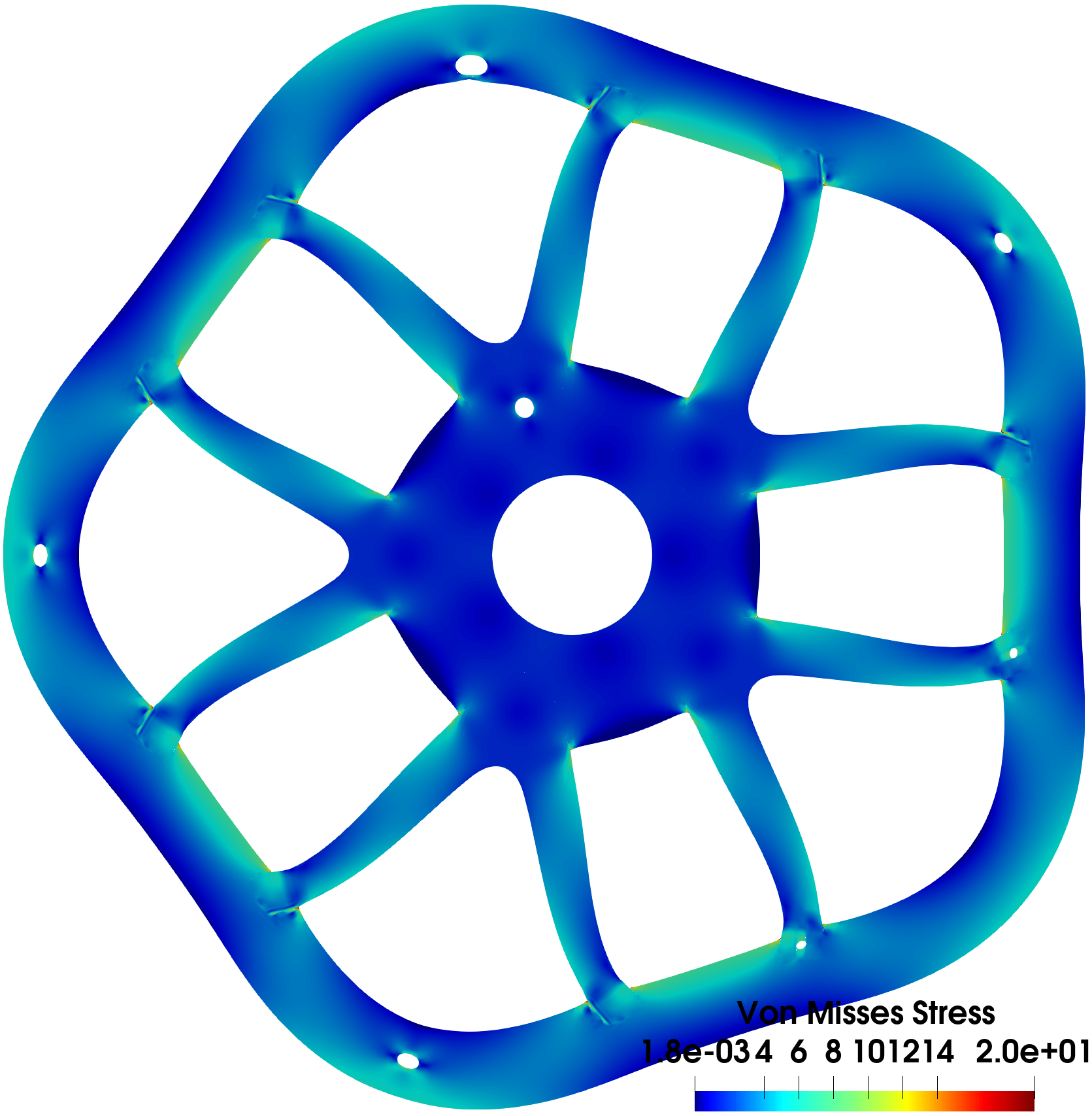} \\
(a) & (b) & (c) \\
\end{tabular}
\caption{A toy car wheel example in the planar domain. (a) Arrangement of patches and boundary conditions, (b) the displacement field (magnitude) on the initial mesh, and (c) the von Mises stress on the mesh after three times of refinement. In (a), $\sigma$ is the stress tensor, $p=1$ is the pressure, $n$ is the outward unit normal, and $\mathbf{u}$ is the displacement vector. The results are visualized on deformed geometries. The white lines represent both the B\'{e}zier mesh and the quadrature mesh for cut elements.}
\label{fig:wheel}
\end{figure}

\section{Conclusion}
\label{sec:con}

We have presented a framework that supports the union operation in isogeometric analysis. As union involves both trimming and interfaces, a so-called minimal stabilization method has been proposed in the context of union to address the stability issue that arises from bad cut elements. We present increasingly complex examples and our results are supported by a comprehensive theory. In the future, we will continue the work in 3D to enable the union operation in isogeometric V-reps. On the other hand, the theoretical study on preconditioning remains an open problem, and it will be another interesting direction to investigate.

\bibliographystyle{siamplain}
\bibliography{ref}
\end{document}